\documentclass[a4paper,11pt,reqno]{amsart}

\usepackage{color}

\usepackage{amsmath,amstext,amssymb,amsopn,amsthm}
\usepackage{url}

\usepackage[margin=30mm]{geometry}
\usepackage{eucal,mathrsfs,dsfont}
\usepackage{soul}
\usepackage{color}
\definecolor{mr}{rgb}{0.1,0.2,0.7}
\definecolor{tk}{rgb}{0.7,0.1,0.2}

\newtheorem{theorem}{Theorem}[section]
\newtheorem{corollary}[theorem]{Corollary}
\newtheorem{lemma}[theorem]{Lemma}

\theoremstyle{definition}

\theoremstyle{remark}
\newtheorem{remark}[theorem]{Remark}
\newtheorem{example}[theorem]{Example}

\newcommand{\eps}{\varepsilon}

\newcommand{\calB}{\mathcal{B}}

\newcommand{\C}{{\mathds{C}}}
\newcommand{\R}{\mathds{R}}
\newcommand{\Rd}{\mathds{R}^d}
\newcommand{\N}{{\mathds{N}}}

\newcommand{\E}{\mathbb{E}}

\newcommand{\wt}{\widetilde}

\title[]{On weak solution of SDE driven by inhomogeneous singular L\'evy noise}

\author[T. Kulczycki]{Tadeusz Kulczycki}
\author[A. Kulik]{Alexei Kulik}
\author[M. Ryznar]{Micha{\l} Ryznar}

\thanks{T. Kulczycki and M. Ryznar were supported in part by the National Science Centre, Poland, grant no. 2019/35/B/ST1/01633}
\address{Faculty of Pure and Applied Mathematics, Wroc{\l}aw University of Science and Technology, Wyb. Wyspia{\'n}skiego 27, 50-370 Wroc{\l}aw, Poland.}
\thanks{A. Kulik has been supported through
the DFG-NCN Beethoven Classic 3 programme, contract no.
2018/31/G/ST1/02252 (National Science Center, Poland) and SCHI-419/11–1
(DFG, Germany)}
\email{tadeusz.kulczycki@pwr.edu.pl}
\email{oleksii.kulyk@pwr.edu.pl}
\email{michal.ryznar@pwr.edu.pl}

\begin{document}

\begin{abstract} We study a time-inhomogeneous  SDE in $\R^d$ driven by a cylindrical L\'evy process with independent coordinates which may have different scaling properties. Such a structure of the driving noise makes it strongly spatially inhomogeneous and complicates the analysis of the model significantly. We prove that the weak solution to the SDE is uniquely defined, is Markov, and has the strong Feller property. The heat kernel of the process is presented as a combination of an explicit `principal part'  and a `residual part', subject to certain $L^\infty(dx)\otimes L^1(dy)$ and $L^\infty(dx)\otimes L^\infty(dy)$-estimates showing that this part is negligible  in a short time, in a sense. The main tool of the construction is the analytic parametrix method, specially adapted to L\'evy-type generators with strong spatial inhomogeneities.
\end{abstract}

\maketitle

\section{Introduction}
In this paper we study an SDE of the form
\begin{equation}
\label{main}
d X_t = \int V_t(X_{t-},z)\, N(dt,dz),
\quad X_{0} = x \in \R^d, \quad t\ge 0,
\end{equation}
where $N(dt,dz)$ is a Poisson random measure, which corresponds to a symmetric L\'evy process $Z = (Z_t, t \ge 0)$ in the usual sense that
$$
d Z_t = \int z \,  N(dt,dz).
$$
Heuristically, the dynamics of the process $X$ can be described as follows: whenever the driving process has a jump with the altitude $\triangle_tZ=z$, the process $X$ makes the jump with the altitude $\triangle_tX=V_t(X_{t-},z)$. Such a description can be made rigorous either if the total intensity of jumps for $Z$ is finite (and then the jumps can be processed one by one), or the jump coefficient $V_t(x,z)$ satisfies a proper version of the Lipschitz condition w.r.t. $x$ (and then the solution to \eqref{main} can be obtained by the It\^o-L\'evy stochastic calculus tools, e.g. \cite[Section~IV.9]{IW81}). In both these cases, $X$ is a \emph{strong} solution to \eqref{main}, i.e. a process adapted to the natural filtration generated by the L\'evy noise. In the current paper, we deal with  a more sophisticated setting where the coefficient is assumed to be H\"older continuous, only. In this case,  one can still expect to have $X$ uniquely defined in law as a \emph{weak} solution to \eqref{main}. The guideline here is provided by the classic diffusion theory \cite{Stroock_Varad}, based on an analytic study of the backward Kolmogorov equation for the (formal) generator, associated with the SDE. Extension of this analytic theory to L\'evy driven SDEs has been a subject of intensive studies, see the literature overview in Section \ref{s24} below. Such an extension is far from being straightforward; namely, because of high diversity of the possible structure of the L\'evy noise,  numerous new effects appear, often requiring specific methods to be treated. In the current paper we approach a quite challenging case, where the driving process $Z$ has the form
\begin{equation}\label{Z}
Z=(Z^1, \dots Z^d),
\end{equation}
with  $Z^i, i=1,\dots, d$ being independent scalar L\'evy processes which have the \emph{weak scaling property} (WSP), see \eqref{h-scaling_u0} below. The jump coefficient will be assumed to have a natural form
\begin{equation}\label{V_t}
V_t(x,z)=A_t(x)z+U_t(x,z)
\end{equation}
with the linear part $A_t(x)z$ being principal, in a sense, for small $|z|$. Clearly, when $U \equiv 0$ equation (\ref{main}) is equivalent to
$$
d X_t = A_t(X_{t-}) \, d Z_t, \quad X_{0} = x \in \R^d.
$$
We stress that even the case of $A_t(x)=A(x)$,  $U_t(x,z)\equiv 0$ and all $Z^i, i=1,\dots, d$ having the same $\alpha$-stable distribution is quite complicated; for instance,  corresponding transition probability densities may fail to be locally bounded. Such an effect appears if  the distributions of a jump for various starting points are mutually singular; for a detailed discussion we refer to \cite[Section~4]{KKS2020}, where such models are called \emph{essentially singular.} The essential singularity in the above setting is caused by a combination of two features:  the fact that the L\'evy measure of the process \eqref{Z} is supported by the collection of the coordinate axes in $\R^d$ and thus is singular w.r.t. the Lebesgue measure, and a non-trivial rotation provided by the matrix $A(x)$.  In this paper we will make one more substantial step further and allow the one-dimensional components of the noise to have \emph{different} laws. To outline the new difficulties which appear in this setting, let us consider for a while $Z$ with $\alpha_i$-stable components, $i=1, \dots, d$. For small $t$, the law of $Z_t$ is mainly concentrated around the axis with the number $j=\mathrm{argmin}_i\,\alpha_i$, which combined with a non-trivial rotation makes the model quite difficult to analyze analytically.

 The first steps in the study of essentially singular models have been made in \cite{KRS18}, \cite{KR19}, \cite{KKS2020} {and \cite{CHZ2020}}. In \cite{KRS18}, the components of the noise were the same and $\alpha$-stable. The results of \cite{KRS18} were significantly extended in \cite{CHZ2020}, where time-inhomogeneous model with a drift was studied. In \cite{KKS2020} general \emph{stable-like} models have been treated, where the stability index and the spherical kernel (i.e. the distribution of the jump direction) are $x$-dependent. In \cite{KR19}, instead of stable noise,  a more general class of noises has been treated, satisfying \emph{weak scaling condition}; see definition in Section \ref{s2} below.  In this paper we extend these previous results in several directions. First, in the setting of \cite{KR19}, where the cylindrical noise  has the {same} laws of the coordinates, we remove several hidden limitations. Namely,
\begin{itemize}
  \item instead of the linear-in-$z$ coefficient $V(x,z)=A(x)z$, we consider the coefficients of the form \eqref{V_t} with a principal linear part and residual non-linearity;
  \item time-inhomogeneous models are engaged into study;
  \item instead of the Lipschitz continuity of the matrix coefficient $A(x)$, the H\"older continuity is assumed.
\end{itemize}
Second, we make a further substantial step, treating a cylindrical noise which has \emph{different} laws of the coordinates. As we have explained before, such an extension leads to substantial analytical difficulties; in addition,  quite new effects  may appear because of different scaling  for various coordinates. Namely, we will see in Example \ref{ex2} that, in this setting,  non-trivial assumptions on the H\"older indices of the coefficients should be made, in the striking contrast to the case of same coordinates, or the stable-like case studied in  \cite{KKS2020}.

To provide a comprehensive analysis of the new effects which appear due to strongly inhomogeneous and singular L\'evy noise, we restrict ourselves to models which do not contain a drift term; i.e. without a gradient term in the generator. Adding a drift term can lead to further complications because of possible lack of domination property in the case of the lower scaling index $\alpha<1$. It is visible that these problems can be resolved by the `flow corrector' method introduced in \cite{KK18}, \cite{Ku18}, see also a discussion in \cite[Sections~6.1,6.2]{KKS2020}; such an extension is a topic of our ongoing research.

 We will prove  existence and uniqueness of the weak solution to \eqref{main}, which will be shown to be a time-inhomogeneous Markov process. We will also provide a representation of the transition probability density of this process as a sum of explicitly given `principal part', and a `residual part' subject to a set of estimates showing that this part is negligible  in a short time, in a sense. The `principal part' will be given in the form
\begin{equation}\label{Euler}
\wt p_{t,s}(x,y)=\frac{1}{|\det A_t(x)|}\wt {G}_{s-t}((y-x)(A_t(x)^{-1})^T), \quad 0\leq t<s,\quad  x,y\in \R^d,
\end{equation}
where $\wt {G}_t(\cdot)$ is the distribution density of $Z_t$. Clearly,  as a function of $y$, this is the distribution density of the variable
\begin{equation}\label{Euler_X}
\wt X_s^{t,x}=x+A_t(x)(Z_s-Z_t),
\end{equation}
which can be seen as a natural approximation to the value at the time instant $s$  of the solution to \eqref{main}, which starts from the point $x$ at the time instant $t$.

It is worth mentioning that recently the existence of densities for SDEs driven by
singular L\'evy processes have been studied in \cite{FJR18} (cf. also \cite{DF13}).

The structure of the rest of the paper is the following. In Section \ref{s2} we introduce the assumptions, formulate the main results, and provide a comprehensive discussion for them,  based on examples and an overview of related results, available in the literature. Sections \ref{s3} --  \ref{Parametrix construction} contain the proofs. The proofs are rather technical, hence in order to improve readability we first explain the keystones  of the proofs in  Section \ref{s3}. Numerous estimates required in the main proof are deduced in Sections \ref{1dim} -- \ref{Parametrix construction}.

\section{Main results}\label{s2}
\subsection{Assumptions} In this section, we collect all the assumptions we impose on our model.
 Let us begin with the description of scalar L\'evy processes involved, as the coordinates, into the representation \eqref{Z}.
Let the characteristic exponent $\psi$ of a one-dimensional, symmetric L\'evy  process be given by
$\displaystyle{\psi(\xi) = \int_{\R} (1 - \cos(\xi x)) \nu(dx) }$,
where $\nu$ is a  symmetric, infinite L{\'e}vy measure. The corresponding  \emph{Pruitt function}  $h(r)$ is given by
$$
h(r) = \int_{\R} (1 \wedge (|x|^2 r^{-2})) \nu(d x), \quad r > 0.
$$





We will assume the following scaling conditions for the function $h$: for some $0<\alpha\leq \beta\leq 2$ and $0<C_1\le 1\le C_2<\infty$,
\begin{equation}
\label{h-scaling_u0}
C_1 {\lambda}^{-\alpha} h\left({r}\right) \le h\left(\lambda r\right) \le
C_2 {\lambda}^{-\beta} h\left({r}\right), \  0<r \le 1,\ 0<\lambda \le 1.
\end{equation}

 We claim that the above assumption is equivalent to the following weak scaling property for $\psi$: there are constants $0<C_1^*\le 1\le C_2^*<\infty$,

\begin{equation}
\label{psi-scaling}
C_1^* {\lambda}^{\alpha} \psi\left({\xi}\right) \le \psi\left(\lambda \xi\right) \le
C_2^* {\lambda}^{\beta}  \psi\left({\xi}\right), \  |\xi| \ge 1,\ \lambda \ge 1.
\end{equation}
The argument of the equivalence is postponed to Section \ref{1dim}.

Once the condition  \eqref{h-scaling_u0} (or equivalently \eqref{psi-scaling})  is  satisfied, we say that the characteristic exponent $\psi$ (or the L\'evy measure $\nu$) have the weak scaling property with indices $\alpha, \beta$, and write $\psi\in \mathrm{WSC}(\alpha, \beta)$ (resp. $\nu\in \mathrm{WSC}(\alpha, \beta)$).

By $\psi_i$, $\nu_i$ and $h_i$ we denote corresponding characteristic exponents, L{\'e}vy measures and Pruitt functions of coordinates $Z^i$ of the process $Z = (Z^1, \ldots,Z^d)$.

We will consider two cases:
\vskip 10pt\noindent
\textbf{(A)} All characteristic exponents $\psi_i, i=1,\dots, d$ are equal and $\psi_1 \in \mathrm{WSC}(\alpha, \beta)$.

\vskip 10pt\noindent\textbf{(B)} Characteristic exponents $\psi_i, i=1,\dots, d$ are not the same and $\psi_i\in \mathrm{WSC}(\alpha, \beta),i=1,\dots, d$.
\vskip 10pt

In both of these cases, the process $Z$ has the transition density $\wt {G}_t(x,y)=\wt {G}_t(y-x),$
where
$$
\wt {G}_t(w)=\prod_{i=1}^{d}\wt {g}_t^i(w_i), \quad w=(w_1,\dots, w_d)\in \R^d,
$$
and $\wt{g}_t^i, i=1,\dots,d$ are the distribution densities for the coordinates (all $\wt{g}^i_t$ are the same in the case \textbf{(A)}).

Next, we assume the following conditions on the coefficients.

\vskip 10pt\noindent \textbf{(C)} For any $t \ge 0$, $x \in \R^d$ $A_t(x) = (a_{t,i,j}(x))$ is a $d \times d$ matrix and there are constants $C_3, \dots, C_6 > 0$, $\gamma_1, \gamma_2 \in (0,1]$ such that for any $s, t \geq  0$, $x, y \in \R^d$, $i, j \in \{1,\ldots,d\}$,
\begin{equation}
\label{bounded}
|a_{t,i,j}(x)| \le C_3,
\end{equation}
\begin{equation}
\label{determinant}
|\det(A_t(x))| \ge C_4,
\end{equation}
\begin{equation}
\label{Holder}
|a_{t,i,j}(x) - a_{t,i,j}(y)| \le C_5 |x - y|^{\gamma_1},
\end{equation}
\begin{equation}
\label{Lipschitz}
|a_{s,i,j}(x) - a_{t,i,j}(x)| \le C_6 |s - t|^{\gamma_2}.
\end{equation}
The function $((0,\infty)\times\R^d\times\R^d) \ni (t,x,z) \to U_t(x,z) \in \R^d$ is continuous and there are constants $C_7 > 0$ and
\begin{equation}
\label{gamma_3}\gamma_3 >\max(1, \beta)
\end{equation}
such that for any $t \geq  0$, $x, z \in \R^d$
\begin{equation}
\label{kernel_E}
|U_t(x,z)| \le C_7 |z|^{\gamma_3}.
\end{equation}
\vskip 10pt
In the case \textbf{(A)}, the H\"older indices $\gamma_1, \gamma_2$ can be arbitrarily small. In the case \textbf{(B)}, these indices and  the $U$-smallness index $\gamma_3$ should satisfy  certain additional assumptions. Namely, we assume the following

\vskip 10pt\noindent\textbf{(D)} \begin{equation}
\label{indices}
\frac{\beta}{\alpha} < 1 + \gamma_1, \quad \quad
\frac{1}{\alpha} - \frac{1}{\beta} < \gamma_2, \quad \quad
\frac{\beta}{\alpha}<  \gamma_3.
\end{equation}

For abbreviation, for any $u > 0$ we will use the notation
$$
\kappa(u) = (u, h_1(1), \ldots, h_d(1),  h^{-1}_1(1), \ldots, h^{-1}_d(1), h_1^{-1}(1/u), \ldots, h_d^{-1}(1/u)).
$$

\subsection{Main statements} In this section, we formulate the main statements of the paper.

\begin{theorem}
\label{main_thm} Assume either \textbf{(A)},\textbf{(C)}, or \textbf{(B)},\textbf{(C),(D)}. Then for any $ x\in \R^d$ the SDE (\ref{main}) has a unique weak solution $X$. The process $X$ is a time-inhomogeneous Markov process which has a transition density $p_{t,s}(x,y)$. The transition density admits a representation
\begin{equation}\label{representation}
p_{t,s}(x,y)=\wt p_{t,s}(x,y)+\wt r_{t,s}(x,y),\quad x,y\in \R^d, \quad 0\le t<s,
\end{equation}
where $\wt p_{t,s}(x,y)$ is given by \eqref{Euler} and the residual part $\wt r_{t,s}(x,y)$ satisfies
$$
\int_{\R^d} |\wt r_{t,s}(x,y)| \, dy \leq c(s-t)^{\eps_0}, \quad x\in \R^d,
$$
where $\eps_0$ is defined in Remark \ref{remark_epsilon} and
the constant $c$ depends only on $d$, $\alpha$, $\beta$, $\gamma_1$, $\gamma_2$, $\gamma_3$, $C_1, \ldots, C_7$, $h_1(1), \dots, h_d(1)$.
\end{theorem}

Theorem \ref{main_thm} actually states that the distribution density for $X_s$ conditioned by $X_t=x$ can be approximated by the density of the variable \eqref{Euler_X}, with the error of approximation given in the integral form. A natural question would be to obtain other types of the bounds for the residue, e.g. uniform in $x,y$. It is known that, in the essentially singular setting, the residue can be locally unbounded, see Example \ref{ex1_5} below. Hence, in order to get a uniform bound for the residue, one has to  impose some new intrinsic assumptions. Here we give one such assumption, formulated in the form inspired by the change of measure argument used in \cite{KRS18}. Alternative possibility would be to use a certain integral-in-$x$ condition, similar to (3.17) in \cite{K19} or   (3.17) -- (3.19) in \cite{KKS2020}.

Denote by $\mu$ the L\'evy measure of the process $Z$, and define  $$
T^{t,z}f(x)= f(x+V_t(x,z)).
$$
Assume the following.
\vskip 10pt\noindent
{\bf{(I)}} For all $t$ and $\mu$-a.a. $z$,  $T^{t,z}$ is a bounded linear operator in $L_{1}(\R^d)$, and there exists $C_8<\infty$ such that
$$
\|T^{t,z}\|_{L_1\to L_1}\leq C_8, \quad t\geq 0, \quad z\in \mathrm{supp}\, \mu.
$$

We have the following representation of the transition density.
\begin{theorem}\label{main_thm2}

Let the conditions of Theorem \ref{main_thm} and additional assumption \textbf{(I)} hold. Then the density
$p_{t,s}(x,y)$ is bounded, that is
$$\sup_{x,y\in \R^d} p_{t,s}(x,y)< \infty, \quad 0\le t<s<\infty.$$
 Moreover,  for any $\tau>0$   there exists  $c>0$, depending only on
 $d$, $\alpha$, $\beta$, $\gamma_1$, $\gamma_2$, $\gamma_3$, $C_1, \ldots, C_8$, $\kappa(\tau)$
such that  the residual term in the representation \eqref{representation} satisfies
$$
|\wt r_{t,s}(x,y)|\leq c\wt {G}_{s-t}(0)(s-t)^{{\eps_0}}, \quad 0< s-t\le \tau,\quad  x,y\in \R^d.
$$
In particular, the following two-sided on-diagonal estimate for $p_{t,s}(x,y)$ holds: \newline
   for $0<s-t\le \tau$, $x\in \R^d$,
\begin{equation}\label{on-diagonal}
 \frac1{|\det A_t(x)|}\left(1-c(s-t)^{\eps_0}\right) \leq \frac{p_{t,s}(x,x)}{\wt {G}_{s-t}(0)}\leq \frac1{|\det A_t(x)|}\left(1+c(s-t)^{\eps_0}\right).
\end{equation}
\end{theorem}

Define  by $\{P_{t,s}\}$ the \emph{evolutionary family} corresponding to the process $X$ in the usual way: for any $0\le t<s$, $x\in \R^d$ and a bounded Borel function $f: \R^d \to \R$,
$$
P_{t,s} f(x) = \int_{\R^d} p_{t,s}(x,y) f(y) \, dy.
$$
Under just the basic conditions of Theorem \ref{main_thm}, we prove  H\"older continuity of this evolutionary family.
\begin{theorem}
\label{Holder_thm_new}
Assume either \textbf{(A)},\textbf{(C)}, or \textbf{(B)},\textbf{(C),(D)}.
For any $0 < \gamma < \gamma' < \alpha$, $\gamma \le 1$,  $0<s-t\le \tau$,  $x, y \in \Rd$ and a bounded Borel function $f: \R^d \to \R$ we have
$$
\left|P_{t,s} f(x) - P_{t,s}f(y)\right|
\le c |x - y|^{\gamma} (s-t)^{-\gamma'/\alpha} \|f\|_{\infty},  \quad 0\le t<s<\infty,
$$
where $c$ depends only on $\gamma$, $\gamma'$, $d$, $\alpha$, $\beta$, $\gamma_1$, $\gamma_2$, $\gamma_3$, $C_1, \ldots, C_7$, $\kappa(\tau)$.
\end{theorem}
\subsection{Examples}   Let us give several examples illustrating various specific issues  of the model.   Our first example shows that the distributions of the components of the L\'evy noise $Z$ can be quite singular. Note that a simplest  example of a L\'evy measure $\nu\in \mathrm{WSC}(\alpha, \beta)$  is a symmetric $\alpha$-stable L\'evy measure
\begin{equation}\label{alpha}
\nu(dx)=c\frac{dx}{|x|^{\alpha+1}},
\end{equation}
for which
$$
h(r)=\frac{4c}{\alpha(2-\alpha)} r^{-\alpha}
$$
and thus \eqref{h-scaling_u0} holds true with $\beta=\alpha$ and $C_1=C_2=1$. The weak scaling property has the same spirit with the (true) scaling property of the $\alpha$-stable L\'evy measure, but is much more flexible.

\begin{example}\label{ex1}(Discretized $\alpha$-stable measure) Let $\mu(dx)$ be obtained from the symmetric $\alpha$-stable measure \eqref{alpha} by discretization in the following way:
$$
\mu(dx)=\sum_{k=1}^\infty\frac{\nu(\{y:\rho_{k+1}<|y|\leq \rho_k\})}{2}\Big(\delta_{-\rho_k}(dx)+\delta_{\rho_k}(dx)\Big),
$$
where $\rho_k\searrow 0$ is a given sequence. Assume that $\{\rho_k\}$ decays not faster than geometrically; that is, for some $c>0$
$$
\rho_{k+1}\geq c \rho_k, \quad k\geq 1.
$$
Then it is easy to show that the Pruitt function for $\mu$ satisfies
\begin{equation}\label{h-discrete}
B_1r^{-\alpha}\leq h(r)\leq B_2 r^{-\alpha},\quad r\in (0,1],
\end{equation}
for the reader's convenience we prove this inequality in Appendix \ref{sA} below. This inequality yields immediately that
 the discretized measure $\mu$ belongs to the same class  $\mathrm{WSC}(\alpha, \alpha)$ with  the original $\alpha$-stable measure.
\end{example}

The following two examples illustrate the difference between the integral-in-$y$ estimate for the residual term $r_{t,s}(x,y)$ from Theorem \ref{main_thm} and the uniform estimate for this term from Theorem \ref{main_thm2}. First, we note that, under just the basic assumptions of Theorem \ref{main}, the transition probability density may be locally unbounded.

\begin{example}\label{ex1_5} (See \cite[Remark~4.23]{KRS18}, \cite[Example~4.2]{KKS2020}).
  Let $d>1$, all the coordinates $Z^i, i=1, \dots, d $ have the same $\alpha$-stable distribution, and
  $
  V_t(x,z)=A(x)z$, where  the matrices $A(x)$ are H\"older continuous in $x$, for each $x \in \R^d$ the matrix
  $A(x)$ is a rotation (hence, an isometry) and for any $x$ in some open cone with vertex at $0$, which satisfies $|x| \ge 1$ we have $A(x)\mathbf{e}_1 = x/|x|$. Then for $\alpha+1\leq d$, for any $x\in \R^d$ the transition probability density $p_t(x,y)$ is unbounded at any neighbourhood of the point $y=0$.
\end{example}

In the above example, an `accumulation of mass' effect appears due to singularity of the noise combined with a non-trivial rotation provided by the matrix $A(x)$. The next example shows a typical situation where the additional condition \textbf{(I)} holds true, and thus the  `accumulation of mass' effect does not appear.

\begin{example}\label{ex1_6} Let the function $V_t(x,z)$ be Lipschitz continuous in $x$ with the Lipshitz constant satisfying
$$
\mathrm{Lip}\,\Big( V_t(\cdot,z)\Big)\leq \rho, \quad t\geq 0, \quad z\in \R^d,
$$
where $\rho<1$. Then the mapping $I_{\R^d}+ V_t(\cdot,z)$ has an inverse and

$$
\mathrm{Lip}\,\Big(\left[I_{\R^d} + V_t(\cdot,z)\right]^{-1}\Big)\leq \frac1{1-\rho}, \quad t\geq 0, \quad z\in \R^d.
$$
Moreover $\left[I_{\R^d} + V_t(\cdot,z)\right]^{-1}$
has a gradient, which is defined a.e. with respect to the Lebesgue measure and bounded, see \cite{C1976}.  In addition, the following change of variables formula holds \cite{H1993}:
$$
\int_{\R^d} f(x+V_t(x, z))\, dx=\int_{\R^d} f(v)  \det (\nabla_v[I_{\R^d}+V_t(v,z)]^{-1})\, dv.
$$
This yields  \textbf{(I)} with
$$
C_8=\sup_{t,z}\mathop{\mathrm{esssup}}\limits_v\,| \det (\nabla_v[I_{\R^d}+V_t(v,z)]^{-1})|\le \frac{d!}{(1-\rho)^d} .
$$

\end{example}

 Our last example explains why in the case \textbf{(B)}, i.e. for a cylindrical noise which has \emph{different} scaling indices of the coordinates,  non-trivial assumptions on the H\"older indices of the coefficients should be made, on the contrary to the case \textbf{(A)}, where the  H\"older indices can be arbitrarily small.

\begin{example}\label{ex2}
  Let $Z^i,i=1, \dots,d$ be symmetric $\alpha_i$-stable with different $\alpha_i, i=1, \dots, d$. The process $Z=(Z^1, \dots,Z^d)$ fits to our case  \textbf{(B)} with $\alpha=\min_i\alpha_i$, $\beta=\max_i\alpha_i$. In this example, we show that  in such -- extremely spatially non-homogeneous -- setting the additional assumption \textbf{(D)} is crucial in the sense that, without this condition, the structure of the transition density can be quite different.

  Take $d=2$ and $\alpha_1=\alpha<\alpha_2=\beta$. Take also $V_t(x,z)=A_t z$, where
  $$
 A_t=\left(
                                        \begin{array}{cc}
                                          1 & t^\gamma \\
                                          t^{\gamma} & 1 \\
                                        \end{array}
                                      \right)
  $$
is a matrix-valued function which depends on $t$, only. Then the additional assumption  \textbf{(I)} holds true since each  operator $T^{t,z}$ is just an isometry which corresponds to the shift of the variable $x\mapsto x+A_tz$ (we can also refer to Example \ref{ex1_6} here). Denote, as usual, $f\asymp g$ if the ratio $\frac{f}{g}$ is bounded and separated from $0$. Then, by Theorem \ref{main_thm2}, one has
$$
p_{t,s}(x,x)\asymp \wt G_{s-t}(0)  \asymp (s-t)^{-\frac{1}{\alpha}-\frac{1}{\beta}}
$$
provided that $\frac{1}{\alpha} - \frac{1}{\beta} < \gamma$, which is actually the second inequality in \eqref{indices} (the first and the third one hold true automatically). For $\frac{1}{\alpha} - \frac{1}{\beta} > \gamma$ the situation changes drastically; namely, we have
\begin{equation}\label{on-diagonal-est}
p_{0,s}(x,x)\leq C s^{-\gamma-\frac{2}{\beta}} \quad \hbox{and} \quad  \frac{s^{-\gamma-\frac{2}{\beta}}}{s^{-\frac{1}{\alpha}-\frac{1}{\beta}}}\to 0, \quad s\to 0+.
\end{equation}
We prove this relation in Appendix \ref{sA}; here we give an informal explanation of the effect. The original noise has two components, a `weaker' one  and a `stronger' one, which act along the 1st and the 2nd coordinate vectors $\mathbf{e}_1, \mathbf{e}_2$, respectively. The law of the solution to SDE \eqref{main} with $x=0, t=0$ is a convolution of the laws of the solutions $X^{(1)}, X^{(2)}$ to  SDE \eqref{main}, where instead of $Z$ we substitute these two components of the noise separately. Consider the projections of these laws on the direction $\mathbf{e}_1$, the one where the `weaker' noise acts. It is easy to show that the projection of the law of $X^{(1)}_s$ on $\mathbf{e}_1$ has a distribution density $p_{0,s}^{(1,1)}(0,\cdot)$ with $p_{0,s}^{(1,1)}(0,0)\asymp s^{-\frac{1}{\alpha}}.$ On the other hand, any jump of the noise at the time $t$, having the altitude $z$ and the direction $\mathbf{e}_2,$ produces a jump of the 1st coordinate with the altitude $t^\gamma z$. Then it is not difficult to prove (and it is easy to believe) that the  projection of the law of $X^{(2)}_s$ on $\mathbf{e}_1$ has a distribution density $p_{0,s}^{(2,1)}(0,\cdot)$ with $p_{0,s}^{(2,1)}(0,0)\asymp s^{-\gamma-\frac{1}{\beta}}$. Since $-\gamma-\frac{1}{\beta}> -\frac{1}{\alpha}$, this means that the projection of the law of $X^{(1)}_s$ in the direction $\mathbf{e}_1$ is `more concentrated' around $0$ than the same projection for the law of  $X^{(2)}_s$. The direction $\mathbf{e}_1$ is the `worst possible' here in the sense that $A_0\mathbf{e}_1$ is equal to the first basis vector $\mathbf{e}_1$ and is orthogonal to the second one $\mathbf{e}_2$.
One can actually show that the same `concentration comparison' hold true for the projections on arbitrary direction $\mathbf{l}$. This means that, in the convolution of the laws $X^{(1)}, X^{(2)}$, the first component is negligible when compared to the second one. It should be noted that, in this example, the one-dimensional noise $Z^2\mathbf{e}_2$ generates a two-dimensional distribution density, which is actually a hypoellipticity-type effect. This density appears to be principal for the entire solution, which indicates that, without a condition of the type \textbf{(D)},  analysis of the SDE with different components of the cylindrical  noise should involve a study of hypoellipticity features. A systematic study of that type  does not seem realistic for SDEs with low regularity of the coefficients, thus we restrict ourselves to the case where  the condition \textbf{(D)} holds and thus hypoellipticity-type effects do not come into play.
\end{example}

\subsection{Literature overview}\label{s24} Our main tool in the construction of the heat kernel $p_{t,s}(x,y)$ of the solution to SDE \eqref{1}
is the \emph{parametrix method}, properly adapted to the sophisticated model we have. The parametrix method was first proposed by Levi~\cite{Le1907},   Hadamard~\cite{had1911} and Gevrey~\cite{gev13} for differential operators  and later extended  by Feller~\cite{Fe36} to a simple non-local setting. The first version of the parametrix method for non-local operators  was developed  by Kochubei~\cite{K89}, see also the monograph by  Eidelman, Ivasyshen \& Kochubei~\cite{EIK04}. This method required the L\'evy measure of the noise to be comparable with the rotationally invariant $\alpha$-stable L\'evy measure, and $\alpha>1$, i.e. the non-local part of the generator should dominate -- in the order sense -- the gradient part. These results have been extended in numerous directions e.g. by  Kolokoltsov~\cite{Ko00}, where the limitation $\alpha>1$ was removed for the operators without a gradient part; see also Chen \& Zhang~\cite{CZ16}. The parametrix method for the \emph{stable-like} case, where the \emph{stability index} is $x$-dependent, have been developed first by Kolokoltsov~\cite{Ko00}; in the papers of K\"uhn~\cite{Kue17b,Kue17a} this problem was treated for a wider class of L\'evy kernels assuming a kind of sector condition for the symbol of the operators. In Knopova \& Kulik~\cite{KK18,Ku18} the parametrix method was extended to \emph{super-critical} case, where the (non-trivial) gradient part is not dominated by an $\alpha$-stable noise with $\alpha<1$. In all these results the L\'evy noise, principally, was comparable with the rotationally invariant $\alpha$-stable one. L\'evy-type models with other types of the reference measures have been studied as well; see Bogdan, Knopova \& Sztonyk~\cite{BKS17}, Kulczycki \&  Ryznar ~\cite{KR17}, where $\alpha$-stable reference measures with various types of spherical  measure (i.e. the distribution of the jumps directions) have been treated, and Grzywny \& Szczypkowksi \cite{GS19}, where the reference measure is rotationally invariant and satisfies weak scaling condition. The \emph{symmetry} assumption, typically imposed on the L\'evy noise in order to simplify the technicalities, is not substantial; see the recent publications  by Chen, Hu, Xie \& Zhang \cite{CZ18, CHXZ17}, Grzywny \& Szczypkowksi \cite{GS19}, Kulik \cite{K19} for the parametrix method for various non-symmetric L\'evy-type models.

\emph{Essentially singular} models, where the distributions of a jump for various starting points are mutually singular, lack a fixed reference measure, to the striking contrast with the results listed above. This leads to a considerably new technical difficulties;  essentially singular models also exhibit new effects such as the one discussed in Example \ref{ex1_5}. For the first advances in the study of such models see Kulczycki, Ryznar \& Sztonyk~\cite{KRS18,KR19} and Knopova, Kulik \& Schilling \cite{KKS2020}, which we have already mentioned and discussed in the Introduction.

\section{Road map to the proofs}\label{s3}
\subsection{The parametrix method}\label{s31} We will construct the transition density $p_{t,s}(x,y)$ for the unknown process using a proper modification of the \emph{parametrix method}, which is a classical analytical method for construction of \emph{fundamental solutions} to elliptic and parabolic PDEs of second order; for a detailed overview of the history and the ideas the method is based on, we refer to \cite{KKK19} or \cite{KKS2020}. Here we outline briefly the construction, taking into account the fact that the actual model is non-homogeneous in time.

Consider a (time-dependent) operator $L_t$ with the domain $C_\infty^2(\R^d)$, given by
$$\begin{aligned}
L_tf(x)&=\mathrm{P.V.}\int_{\R^d}\Big(f(x+V_t(x,z))-f(x)\Big)\mu(d z)
\\&=\sum_{k=1}^d\int_{\R}\Big(f(x+uA_t(x)\mathbf{e}_k+U_t(x,\mathbf{e}_ku))-f(x)
\\&\hspace*{3.5cm}-u1_{|u|\leq 1}\nabla f(x)\cdot A_t(x)\mathbf{e}_k\Big)\nu_k(d u),
\end{aligned}
$$
where $\mu$ is the L\'evy measure of the process $Z$, $\nu_k$ is the L\'evy measure of the $k$-th component $Z^k, k=1, \dots, d$, and $\mathrm{P.V.}$ means that the first integral is taken in the principal value sense. By the virtue of the It\^o formula, one can naturally expect that, once the solution $X$ to \eqref{main} is well defined and is a (time-inhomogeneous) Markov process, the operator $L_t$ should be its generator. Corresponding \emph{Kolmogorov's backward  differential equation} for the transition probability density of $X$ has the form
\begin{equation}\label{backward0}
(\partial_t+L_{t;x})p_{t,s}(x,y)=0, \quad 0\leq t<s, \quad x,y\in \R^d,
\end{equation}
here and below $x$ at the operator $L_{t;x}$ indicates that the operator $L_t$ is applied with respect to the variable $x$. Together with the initial condition
\begin{equation}\label{initial}
p_{t,s}(x,y)\to \delta_x(y), \quad s\to t,
\end{equation}
this actually gives that $p_{t,s}(x,y)$ is a \emph{fundamental solution} to the parabolic Cauchy problem for the operator $L_t$. The streamline of the method is to construct a (candidate for) the required fundamental solution, and then to show that this kernel $p_{t,s}(x,y)$ indeed corresponds to the unique weak solution to \eqref{main}.

To construct a candidate for the fundamental solution, we use \emph{the parametrix method}, which, in a wide generality, can be outlined as follows. Fix a function $p_{t,s}^{(0)}(x,y)$, which is $C^1$ in $t$ and $C_\infty^2(\R^d)$ in $x$ for a fixed $s,y$, and define
$$
q_{t,s}^{(0)}(x,y)=-(\partial_t+L_{t;x})p_{t,s}^{(0)}(x,y).
$$
Then differential equation \eqref{backward0} can be written as
$$
(\partial_t+L_{t;x})(p_{t,s}(x,y)-p_{t,s}^{(0)}(x,y))=q_{t,s}^{(0)}(x,y).
$$
Since we expect $p_{t,s}(x,y)$ to be a (true) fundamental solution, we can formally resolve the above equation as \begin{equation}\label{parametrix_int}
p_{t,s}(x,y)=p_{t,s}^{(0)}(x,y)+\int_t^s\int_{\R^d}p_{t,r}(x,v)q_{r,s}^{(0)}(v,y)\, dvdr, \quad 0\leq t<s, \quad x,y\in \R^d,
\end{equation}
 The identity \eqref{parametrix_int} can be seen as an integral equation for the unknown kernel $p_{t,s}(x,y)$, which is easier to deal with than the original differential equation \eqref{backward0}. This is the essence of the method: we first construct
a candidate for the transition probability density $p_{t,s}(x,y)$ as the solution to the \emph{integral} equation \eqref{parametrix_int} and then  study its properties in order to show that this kernel  indeed corresponds to the unique weak solution to \eqref{main}.

\subsection{Choice of the zero-order approximation}\label{s32} One of the crucial points in the strategy outlined above is the choice of the kernel  $p_{t,s}^{(0)}(x,y)$, which has a natural meaning of the `zero-order approximation' term for the unknown $p_{t,s}(x,y)$. This choice determines the `differential error of approximation' $q_{t,s}^{(0)}(x,y)$, and should be precise  enough to guarantee integrability  of $q_{t,s}^{(0)}(x,y)$; note that we require this integrability  in order to treat the integral equation \eqref{parametrix_int} properly.
We will choose $p_{t,s}^{(0)}(x,y)$ in the form
\begin{equation}\label{p_0}
p_{t,s}^{(0)}(x,y)=\frac{1}{|\det A_s(y)|}G_{s-t}\Big((y-x)(A_s(y)^{-1})^T\Big),
\end{equation}
where $G_r(w)$ is the distribution density of a dynamically truncated L\'evy noise; see Section \ref{Parametrix construction} for its definition and properties.  The density $G_{s-t}$ is also dependent on $\eps>0$, however we do not reflect this in our notation.  Such a choice combines two ideas. The first one is the classical parametrix idea that a good `zero-order approximation' to the fundamental solution can be obtained by taking the heat kernel for an equation with constant coefficients (e.g. the Gaussian kernel in the diffusion setting) and substituting there the coefficients \emph{frozen at the endpoint} $y,s$. This classical construction also suggests that negligible (in a sense) parts should be removed from the generator: in the diffusion setting this is the drift (gradient) term, in our case this is the non-linear jump term $U_t(x,z)$. Though, such classical parametrix construction appears to be not precise enough in the singular L\'evy noise setting. Namely, such a construction would suggest, instead of \eqref{p_0}, the choice
$$
 \wt p_{t,s}^{(0)}(x,y)=\frac{1}{|\det A_s(y)|}\wt G_{s-t}\Big((y-x)(A_s(y)^{-1})^T\Big),
$$
recall that $\wt {G}_t(\cdot)$ is the distribution density of $Z_t$. However, in general,  $\wt p_{t,s}^{(0)}(x,y)$ may provide quite a poor approximation to $p_{t,s}(x,y)$: e.g. in the model from Example \ref{ex2} it can be  verified easily that
$$
\int_{\R^d}\wt p_{t,s}^{(0)}(x,y)\, dy=\infty,
$$
in the striking contrast to the fact that $p_{t,s}(x,\cdot)$ should be a probability density. This is an essentially non-local effect; in order to avoid it we use the second idea to `cut off' big jumps. In \cite{KRS18}, \cite{KR19} the cut off level was chosen small but fixed, which required Lipschitz continuity of the coefficients. In \cite{KKS2020}, a time dependent cut off level was proposed, which allows one to treat the models where the coefficients  are only assumed to be  H\"older continuous.  Here we use the same dynamic truncation idea, properly adapted to the current model. Namely, $G_r(w)$ in \eqref{p_0} will be the distribution density of  $\widehat Z_r=(\widehat Z_r^1, \dots,\widehat Z_r^d)$,  where the components
are independent and
$$
\widehat Z_r^i=\int_0^r\int_{|u|\leq R_\rho^{(i)}} u N^i(d\rho, du),
$$
where $N^i(d\rho, du)$ is the Poisson point measure corresponding to $Z^i$, and the time-dependent truncation function $R_\rho^{(i)}=R_\rho^{(i)}(\eps)$ is determined by means of the corresponding Pruitt function $h_i(r)$ and $\eps>0$ of our choice. Note that in the case  \textbf{(A)} the cut off level is the same for all coordinates, while in the case \textbf{(B)} these levels can be quite different. This is the actual reason for  the condition \textbf{(D)} to appear in the case \textbf{(B)}:  we will need this condition in order to balance, in a sense, the `cut off effects' for various coordinates.

\subsection{Functional analytical framework}\label{s30}
It is  convenient to treat~\eqref{parametrix_int} within the functional analytic framework introduced in \cite[Section~5.2]{KKS2020}, properly adapted to the time non-homogeneous setting. Consider the Banach space  $L^\infty(dx)\otimes L^1(dy)$ of kernels $k(x,y)$ satisfying
\begin{align*}
    \|k\|_{\infty,1}
    := \mathop{\mathrm{esssup}}_{x\in \R^d}\int_{\R^d}|k(x,y)|\, dy<\infty.
\end{align*}
Each kernel $k\in L^\infty(dx)\otimes L^1(dy)$ generates a bounded linear operator $K$ in the space $B_b =  B_b(\R^d)$  of bounded measurable functions,
\begin{align*}
    K f(x) = \int_{\R^d}k(x,y) f(y)\, dy, \quad f\in B_b(\R^d),
\end{align*}
with  the operator norm $\|K\|_{B_b\to B_b}$ equal to the norm $\|k\|_{\infty,1}$. Denote $P_{t,s}, P_{t,s}^{(0)},  Q_{t,s}^{(0)}, 0<t<s$ the families of operators corresponding to the unknown transition probability density $p_{t,s}(x,y)$ and the kernels $p_{t,s}^{(0)}(x,y)$, $q_{t,s}^{(0)}(x,y)$ introduced above. Then~\eqref{parametrix_int} can be equivalently written as
\begin{align}\label{para-e16}
    P_{t,s}
    = P^{(0)}_{t,s}+\int_t^s P_{t,r}Q_{r,s}^{(0)} \, dr, \quad 0\le t<s.
\end{align}
Let $0<\eps\le \eps_0$, with $\eps_0$ defined in  Remark \ref{remark_epsilon}. In the whole subsection $c$ denotes a constant dependent on $d$, $\alpha$, $\beta$, $\gamma_1$, $\gamma_2$, $\gamma_3$, $C_1, \ldots, C_8$, $\kappa(\tau)$ and $\eps$. In Lemma \ref{q0_estimates_lemma} below we prove an estimate for $q_{t,s}^{(0)}(x,y)$ in $\|\cdot\|_{\infty,1}$-norm which actually  can be written as a bound for the operator norm
\begin{align}\label{para-e20}
    \|Q_{t,s}^{(0)}\|_{B_b\to B_b}\leq  c  (s-t)^{-1+\eps}, \quad 0<s-t\leq \tau.
\end{align}
This allows us to treat~\eqref{para-e16}, in a standard way, as a Volterra equation with a mild (integrable) singularity. Recall that each kernel $p_{t,s}(x,y)$ is supposed to be a probability density, hence it is necessary that
\begin{align}\label{para-e22}
    \|P_{t,s}\|_{B_b\to B_b}<\infty .
\end{align}
The unique solution to~\eqref{para-e16} which satisfies~\eqref{para-e22} can be interpreted as a classical Neumann series
\begin{align}\label{para-e24}
\begin{aligned}
    P_{t,s}
    &=P_{t,s}^{(0)}+\sum_{k=1}^\infty\;\;\idotsint\limits_{t<r_1<\dots <r_{k}<s} P_{t, r_1}^{(0)} Q_{r_1,r_2}^{(0)} \dots Q_{r_k,s}^{(0)} \,dr_1\dots dr_{k}\\
    &= P_{t,s}^{(0)} + \int_t^s P_{t,r}^{(0)} Q_{r,s}\, dr,
\end{aligned}
\end{align}
where the operator
\begin{align}\label{para-e25}
    Q_{t,s}
    :=Q_{t,s}^{(0)} + \sum_{k=1}^\infty\;\;\idotsint\limits_{t<r_1<\dots <r_{k}<s} Q_{t,r_1}^{(0)} \dots Q_{r_{k},s}^{(0)} \,dr_1\dots dr_{k}
\end{align}
corresponds to the kernel
\begin{align}\label{para-e26}
    q_{t,s}(x,y) := \sum_{k=0}^\infty q_{t,s}^{(k)}(x,y), \quad  q_{t,s}^{(k+1)}(x,y):=\int_t^s\int_{\R^d}q_{t,r}^{(k)}(x,v)q_{r,s}^{(0)}(v,y)\, dvdr, \quad k\geq 0.
\end{align}
The series~\eqref{para-e25}, \eqref{para-e26} converge uniformly in $0<s-t\leq \tau$ in the operator norm $\|\cdot\|_{B_b\to B_b}$ and the norm $\|\cdot\|_{\infty,1}$, respectively. This follows easily from~\eqref{para-e20}, since for $k\geq 1$
\begin{align}\label{para-e28}
\begin{aligned}
    \|q_{t,s}^{(k)}\|_{\infty,1}
    &= \bigg\|\;\;\idotsint\limits_{t<r_1<\dots <r_{k}<s}Q_{t,r_1}^{(0)} \dots Q_{r_{k},s}^{(0)} \,dr_1\dots dr_{k}\bigg\|_{B_b\to B_b}\\
    &\leq\;\;\idotsint\limits_{t<r_1<\dots <r_{k}<s}\|Q_{t,r_1}^{(0)} \dots Q_{r_{k},s}^{(0)}\|_{B_b\to B_b} \,dr_1\dots dr_{k}\\
    &\leq  c^k  \idotsint\limits_{t<r_1<\dots <r_{k}<s} (r_1-t)^{-1+\epsilon}\cdot\ldots\cdot (s-r_k)^{-1+\epsilon}\,dr_1\dots \, dr_{k}\\
    &=  (s - t)^{-1+k \epsilon} \frac{(c\Gamma(\epsilon))^k}{\Gamma(k\epsilon)},
\end{aligned}
\end{align}
and the Gamma function $\Gamma(z)$ behaves asymptotically like $\sqrt{2\pi} z^{z-\frac 12} e^{-z}\gg c^z$ as $z\to \infty$. This estimate yields
\begin{align}\label{para-e30a}
    \|q_{t,s}\|_{\infty,1}\leq c(s-t)^{-1+\epsilon}, \quad 0<s-t\leq \tau.
\end{align}
In Lemma \ref{py_integrable} below we prove that $p_{t,s}^{(0)}(x,y)$ is bounded in $\|\cdot\|_{\infty,1}$-norm, which similarly to \eqref{para-e28} yields that the residue $r_{t,s}(x,y)=p_{t,s}(x,y)-p_{t,s}^{(0)}(x,y)$ satisfies
\begin{align}\label{para-e30}
    \|r_{t,s}\|_{\infty,1}\leq c(s-t)^{\epsilon}, \quad  0<s-t\leq \tau.
\end{align}
These representations and estimates form  an essential part of the proof of Theorem \ref{main_thm}.

\subsection{Approximate fundamental solution and weak uniqueness of the solution} \label{s33}
Let $C_\infty(\R^d)$ denote  a space of  continuous functions vanishing at infinity.    Define $P_{t,t}=P_{t,t}^{(0)}=\mathrm{id}$, the identity operator. The operator families $\{P_{t,s}^{(0)}\}$, $\{Q_{t,s}^{(0)}\}$ have the following properties, see  Lemmas \ref{py_integrable}, \ref{q0_estimates_lemma}, \ref{q0_new_estimates_lemma}, \ref{pts0_delta} and \ref{qts0_delta}.
\begin{lemma}\label{l41}
   Each of the operators $P_{t,s}^{(0)}, 0\leq t\leq s,  Q_{t,s}^{(0)}, 0\leq t<s$  maps $C_\infty(\R^d)$ to $C_\infty(\R^d)$. The corresponding families of operators are strongly continuous w.r.t. $t,s$.
\end{lemma}
Note that the operator norm $\|\cdot\|_{C_\infty\to C_\infty}$ is dominated by the norm $\|\cdot\|_{B_b\to B_b},$ thus the norm estimates from the previous section yield that the series~\eqref{para-e24} converges in the norm $\|\cdot\|_{C_\infty\to C_\infty}$ uniformly in $0\leq s- t\leq \tau$, and the series \eqref{para-e25} converges uniformly for $ \tau_1\le  s-t\le \tau$, for any $0<\tau_1< \tau$. Moreover, a standard argument based on the strong continuity of $Q_{t,s}^{(0)}, t<s$  and the  estimate \eqref{para-e20} shows that for every $n\ge 1$
 the  operators $Q_{t,s}^{(n)}$ (defined by kernels $q_{t,s}^{(n)}(x,y)$) are strongly continuous w.r.t. $t<s$.
 This yields
\begin{corollary}\label{c41}
  Each of the operators $P_{t,s}, 0\leq t\leq s,  Q_{t,s}, 0\leq t<s$  maps $C_\infty(\R^d)$ to $C_\infty(\R^d)$. Corresponding families of operators are strongly continuous w.r.t. $t,s$.
\end{corollary}

In general, it might be quite difficult to prove that the kernel $p_{t,s}(x,y)$, constructed as a solution to the \emph{integral} equation \eqref{parametrix_int}, solves the \emph{differential} equation \eqref{backward0}. We avoid this complicated step, using the following approximate procedure.
Define for $\eta>0$
$$
P_{t,s, \eta} = P_{t,s+\eta}^{(0)} + \int_t^s P_{t,r+\eta}^{(0)} Q_{r+\eta,s+\eta}\, dr, \quad 0\leq t\leq s.
$$
The following lemma shows that $p_{t,s}(x,y)$ solves the backward Kolmogorov equation \eqref{backward0} in a certain approximate sense.
\begin{lemma}\label{l42} Let  $\tau>0$ and a compact subset $F\subset C_\infty(\R^d)$ be fixed.
  \begin{itemize}
    \item[a)]
    $$
    \|P_{t,s,\eta} f-P_{t,s} f\|\to 0, \quad \eta\to 0
    $$
    uniformly in $0\leq s\leq t\leq \tau, f\in F$,
    \item[b)] For any $f\in C_0(\R^d)$ and $\eta>0$, the function $P_{t,s,\eta} f(x)$ is $C^1$ in $t$ and $C^2_0(\R^d)$ in $x$ on $[0,s]\times \R^d$, and thus the operators
        $$\Delta_{t,s,\eta}=(\partial_t+L_t)P_{t,s, \eta},\quad  s\leq t, \eta>0
               $$
     are well defined. These operators satisfy
   \begin{equation}\label{Delta}
       \|\Delta_{t,s,\eta} f\|\to 0, \quad \eta\to 0
   \end{equation}
    uniformly in $s-t>\tau_1, 0\leq s\leq \tau$, $f\in F$ for any $\tau_1>0$ and
    \begin{equation}\label{Delta_int}
    \int_t^{\tau}\|\Delta_{t,r,\eta} f\|\, dr\to 0, \quad \eta\to 0
    \end{equation}
    uniformly in $0\leq t\leq \tau$, $f\in F$.
  \end{itemize}
 \end{lemma}

The proof of this statement is remarkably simple, and requires only representation \eqref{para-e16} and the continuity properties stated in Lemma \ref{l41} and Corollary \ref{c41}. Thus we give it here.
\begin{proof}  Statement a) follows directly from the representation \eqref{para-e16} and the continuity properties (Lemma \ref{l41} and Corollary \ref{c41}).

  To prove \eqref{Delta}, note that
  $$\begin{aligned}
  \Delta_{t,s,\eta} f&=(\partial_t+L_t)P_{t,s+\eta}^{(0)} f+ \int_t^s (\partial_t+L_t)P_{t,r+\eta}^{(0)} Q_{r+\eta,s+\eta}f\, dr-P_{t,t+\eta}Q_{t+\eta,t+\eta}f
  \\&=Q_{t,s+\eta}^{(0)} f+ \int_t^s Q_{t,r+\eta}^{(0)} Q_{r+\eta,s+\eta}f\, dr-{P}_{t,t+\eta}Q_{t+\eta,s+\eta}f
  \\&=Q_{t,s+\eta}f- \int_t^{t+\eta} Q_{t,r}^{(0)} Q_{r,s+\eta}f\, dr-{P}_{t,t+\eta}Q_{t+\eta,s+\eta}f,
  \end{aligned}
  $$
  in the last identity we have used that $Q_{t,s}$ is given by \eqref{para-e25} and thus satisfies
  $$
  Q_{t,s}=Q_{t,s}^{(0)}+\int_{t}^{s}Q_{t,r}^{(0)} Q_{r,s}\, dr.
  $$
  By the continuity properties Lemma \ref{l41} a),
  $$
  \|Q_{t,s+\eta}f-{P}_{t,t+\eta}Q_{t+\eta,s+\eta}f\|\to 0, \quad \eta\to 0
  $$
  uniformly in $s-t>\tau, f\in F$. Convergence
  $$
  \left\|\int_t^{t+\eta} Q_{t,r}^{(0)} Q_{r,s+\eta}f\, dr\right\|\to 0, \quad \eta\to 0
  $$
  follows by the bounds \eqref{para-e20}, \eqref{para-e30a}.  The same bounds combined with \eqref{Delta} yield \eqref{Delta_int}.
\end{proof}

Lemma \ref{l42} provides an efficient tool for identifying weak solutions to the SDE \eqref{main}. Note that it is easy  to prove existence of a weak solution to \eqref{main} by  smooth approximation of the coefficients and using the compactness argument; see \cite[Section~5]{KK18} for such an argument explained in details. To identify a weak solution to \eqref{main} with given initial condition, we will consider operator $\mathcal{L}$ defined by
$$
\mathcal{L}\phi(t,x)=\partial_t\phi(t,x)+L_{t;x} \phi(t,x), \quad \phi\in \mathcal{D}
$$
on the set $\mathcal{D}=C^{1,2}_\infty([0, \tau]\times \R^d)$ of functions $\phi(t,x)$ which are $C^1$ in $t$, $C^2$ in $x$, and have their derivatives continuous and being from the class $C_\infty(\R^d)$ for any $t$ fixed; $\tau >0$ here is a fixed number. The It\^o formula yields that, for a weak solution $X$ to SDE \eqref{main} with $X_0=x$ and $\phi\in \mathcal{D}$, the process
$$
\phi(s,X_s)-\int_0^s \mathcal{L}\phi(r,X_r)\, dr, \quad s\in [0,\tau]
$$
is a martingale. We can use, with minor changes, the  argument from \cite[Section~5.3]{K19} to derive from this finite-dimensional distributions of $X$. Namely, let $ f\in C_\infty(\R^d), \tau>0$ be fixed.   Taking $\phi(t,x)=P_{t,\tau,\eta}f(x)$, we get for any $t\leq s \leq \tau$
$$
\begin{aligned}
\E\Big[P_{\tau,\tau,\eta}f(X_\tau)\Big|\mathcal{F}_s\Big]-P_{s,\tau,\eta}f(X_s)&=\E\left[\int_s^\tau (\partial_r+L^r_x)P_{t,\tau,\eta}f(X_r)\, dr\Big|\mathcal{F}_s\right]
\\&=\left[\E\int_s^\tau \Delta_{t,\tau,\eta}f(X_r)\, dr\Big|\mathcal{F}_s\right].
\end{aligned}$$
Then by Lemma \ref{l42}, passing to the limit $\eta\to 0$, we get
$$
\E\Big[P_{s,\tau,\eta}f(X_\tau)\Big|\mathcal{F}_s\Big]-P_{s,\tau}f(X_s)=0, \quad s\in [0,\tau]
$$
or, equivalently,
$$
\E\Big[f(X_\tau)\Big|\mathcal{F}_s\Big]=P_{s,\tau}f(X_s)
$$
for any $f\in C_\infty(\R^d)$ and any pair of time moments $\tau\geq s\geq 0$. Since
$$
P_{t,s}f(x)=\int_{\R^d}f(y)p_{t,s}(x,y)\, dy,
$$
this yields the identity
$$
\mathbf{P}(X_{s_1}\in A_1, \dots, X_{s_k}\in A_k)=\iint_{A_1\times\dots \times A_k}p_{0,s_1}(x,v_1)\dots p_{s_{k-1},s_k}(v_{k-1},v_k)\, dv_1 \dots d v_k
$$
valid for any $k\geq 1, t< s_1<\dots<s_k$ and Borel measurable $A_1, \dots, A_k$.
This  identifies  uniquely  the finite-dimensional distributions of $X$ and proves that $X$ is a (time non-homogeneous) Markov process with the transition density $p_{t,s}(x,y)$.

\subsection{Outline of the rest of the proofs}  Recall that our choice of $p_{t,s}^{(0)}(x,y)$ and $q_{t,s}^{(0)}(x,y)$ was dependent on $\eps\in (0, \eps_0]$. In this subsection we choose $\eps= \eps_0$.

To complete the proof of Theorem \ref{main_thm}, we have to prove the bound on the residual term $\wt r_{t,s}(x,y)$ in the decomposition \eqref{representation}. To do this, we use the decomposition
$$
p_{t,s}(x,y)=p_{t,s}^{(0)}(x,y)+r_{t,s}(x,y)
$$
and the bound \eqref{para-e30} for the residual term $r_{t,s}(x,y)$,  obtained by the parametrix method. Then
$$
 \wt r_{t,s}(x,y)=p_{t,s}^{(0)}(x,y)-\wt p_{t,s}(x,y)+r_{t,s}(x,y),
$$
and the required bound follows from the estimate
\begin{align}\label{para-e31}
    \|p_{t,s}^{(0)}(x,y)-\wt p_{t,s}(x,y)\|_{\infty,1}\leq c(s-t)^{\epsilon_0}, \quad 0<s-t\leq \tau,
\end{align}
which we prove in Lemma \ref{p0_ptilde_difference} below.
   Next, combining \eqref{para-e30}  and \eqref{para-e31} we obtain
\begin{align}\label{para-e31a}
    \|\wt r_{t,s}(x,y)\|_{\infty,1}\leq c(s-t)^{\epsilon_0}, \quad 0<s-t\leq \tau,
\end{align}
Here the constant $c$ is a constant dependent on  $d$, $\alpha$, $\beta$, $\gamma_1$, $\gamma_2$, $\gamma_3$, $C_1, \ldots, C_8$, $\kappa(\tau)$, but for $\tau \le \tau_0$, where $\tau_0 $ is defined at the beginning  of Section \ref{Parametrix construction}, the constant depends on $d$, $\alpha$, $\beta$, $\gamma_1$, $\gamma_2$, $\gamma_3$, $C_1, \ldots, C_8$, $h_1(1), \dots, h_d(1)$. In fact (\ref{para-e31a}) holds for all $0\le t<s$ with a constant dependent on  $d$, $\alpha$, $\beta$, $\gamma_1$, $\gamma_2$, $\gamma_3$, $C_1, \ldots, C_8$, $h_1(1), \dots, h_d(1)$.     Indeed, if $s-t> \tau_0$   then
$$\|\wt r_{t,s}(x,y)\|_{\infty,1}=\| p_{t,s}(x,y)-\wt p_{t,s}(x,y)\|_{\infty,1}\leq 2\le \frac 2 {\tau_0^{\epsilon_0}}(s-t)^{\epsilon_0}.$$

The functional analytic framework from Section \ref{s30} is quite convenient also for proving the uniform  estimates for the residual term, stated in Theorem \ref{main_thm2}, see a detailed discussion in \cite[Section~5.2]{KKS2020}. Namely, the uniform-in-$x,y$ bound for a (continuous) kernel is equivalent to the $\|\cdot\|_{L_1\to B_b}$ operator norm for the corresponding integral operator. In addition to the bound \eqref{para-e20}, we have
$$
\|Q_{t,s}^{(0)}\|_{L_1\to B_b}\leq c G_{s-t}(0) (s-t)^{-1+\eps_0}
$$
(Lemma \ref{q0_estimates_lemma}, estimate \eqref{q0_pointwise_estimate1}). Under the additional assumption \textbf{(I)}, we also have
$$
\|Q_{t,s}^{(0)}\|_{L_1\to L_1}\leq c(s-t)^{-1+\eps_0}
$$
(Lemma \ref{q0_estimates_lemma_x}). Then for any $k\in \mathbb{N}, j=0,\dots,k$ and $r_0<r_1<\dots r_{k+1}$
$$
\begin{aligned}
\|Q_{r_0,r_1}^0\dots Q_{r_{k},r_{k+1}}^0\|_{L_1\to B_b}&\leq \|Q_{r_0,r_1}^0\|_{B_b\to B_b}\dots \|Q_{r_{j-1},r_j}^0\|_{B_b\to B_b}
\\&\hspace*{1cm}\times\|Q_{r_j,r_{j+1}}^0\|_{L_1\to B_b}\|Q_{r_{j+1},r_{j+2}}^0\|_{L_1\to L_1}\dots \|Q_{r_{k},r_{k+1}}^0\|_{L_1\to L_1}
\\&\leq c^{k+1} G_{r_{j+1}-r_j}(0)\prod_{i=0}^k(r_{i+1}-r_i)^{-1+\eps_0}.
\end{aligned}$$
If we take $r_0=t, r_{k+1}=t$, and $j$ such that $r_{j+1}-r_j=\max_i(r_{i+1}-r_i)$, then $G_{r_{j+1}-r_j}(0)\leq G_{(s-t)k^{-1}}(0)$, since  $G_{u}(0)$ is nonincreasing as a function of $u>0$. By Corollary \ref{Gu(0)estimate}, we have
$$
G_{(s-t)k^{-1}}(0)\leq c  k^{d/\alpha} G_{s-t}(0), \quad k\in \mathbf{N}.
$$
Also, by Lemma \ref{gu_difference},
$$
 G_{s-t}(0)\le c \wt  G_{s-t}(0), \  0<s-t\le \tau.
$$
Then, similarly to \eqref{para-e28}, \eqref{para-e30a} we get
\begin{align}\label{para-e32}
    \|Q_{t,s}\|_{L_1\to B_b}\leq c\wt  G_{s-t}(0)(s-t)^{-1+\eps_0}, \quad 0<s-t\leq \tau.
\end{align}
We have that $\|P^{(0)}_{t,s}\|_{B_b\to B_b}$  is bounded and, in addition, by (\ref{p_0})
$$
\|P_{t,s}^{(0)}\|_{L_1\to L_\infty}\leq c G_{s-t}(0).
$$
Then using the parametrix representation \eqref{para-e24} and repeating the argument above we get
$$
\|P_{t,s}-P_{t,s}^{(0)}\|_{L_1\to L_\infty}\leq c G_{s-t}(0)(s-t)^{\eps_0}, \quad 0<s-t\leq \tau.
$$
This is actually a uniform-in-$x,y$ bound for the residual term $r_{t,s}(x,y)$ in the decomposition obtained by the parametrix method. To complete
the proof of Theorem \ref{main_thm2}, we prove the corresponding analogue of \eqref{para-e32}, see (\ref{p0_ptilde_difference_max}):
$$
\|\wt P_{t,s}-P_{t,s}^{(0)}\|_{L_1\to L_\infty}\leq c G_{s-t}(0)(s-t)^{\eps_0}, \quad 0<s-t\leq \tau.
$$

 The proof of Theorem \ref{Holder_thm_new} (postponed to Section \ref{Parametrix construction})
stating the H\"older continuity for the evolutionary family $\{P_{t,s}\}$,  is based on the parametrix representation \eqref{para-e16} for this family and the H\"older continuity of the family $\{P_{t,s}^{(0)}\}$ involved in this representation.

\section{One-dimensional density}\label{1dim}
This section is devoted to the study of one-dimensional components of the process $Z = (Z_1,\ldots,Z_d)$. Recall that the characteristic exponent of $Z_i$ is $\psi_i$. In this section we fix $i \in \{1,\ldots,d\}$ and  $\psi$ denotes the fixed  $\psi_i$. By $\nu$, $h$ and $\tilde{g}$ we denote the corresponding L{\'e}vy measure, the Pruitt function and the transition density, respectively. We will construct a truncated version $g$ of the transition density $\tilde{g}$. We will show various estimates of $g$, $\tilde{g}$ and its derivatives. These construction and estimates will play a crucial role to make the parametrix construction in Section \ref{Parametrix construction} work.

For $r > 0$ we put
$$
K(r) = \int_{\{x\in\R: \, |x| \le r\}} |x|^2 r^{-2} \nu(d x) .
$$
We have the following relationship between  $K(r)$ and $h(r)$ \cite[Lemma 2.2]{GS20},

\begin{equation} h(r)= 2\int_r^\infty   K(w)w^{-1} \, dw, \, \, r>0.\label{h_K}\end{equation}

We observe that, due to infinitness of the L\'evy measure $K(w)>0, w>0$, hence  $h(r)$  is strictly decreasing on $(0, \infty)$.

Clearly, $r^2h(r), r^2K(r)$ are increasing on $(0, \infty)$.
Using the monotonicity of the function $r^2h(r)$ we can easily  extend (\ref{h-scaling_u0}) to all $\theta>0$,
\begin{equation}
\label{h-scaling_u5}
C_1 {\lambda}^{-\alpha} h\left({\theta}\right) (1\vee\theta^2)^{-1} \le h\left(\lambda \theta\right) \le
C_2 {\lambda}^{-\beta} h\left({\theta}\right) (1\vee\theta^2), \  \ 0<\lambda \le 1.
\end{equation}

Let us observe that the scaling property (\ref{h-scaling_u0}) is equivalent to
\begin{equation}
\label{h-scaling_11}
 C_1^{1/\alpha}h^{-1}\left({\theta}\right)  {\lambda}^{-1/\alpha}\le  h^{-1}\left(\lambda \theta\right)\le  C_2^{1/\beta}h^{-1}\left({\theta}\right)  {\lambda}^{-1/\beta} ,  \   \theta >h(1),\ \lambda \ge 1.
\end{equation}
Morever this can be extended to all $\theta>0$ (via (\ref{h-scaling_u5})),

\begin{equation}
\label{h-scaling_11+}
 C_1^{1/\alpha}h^{-1}\left({\theta}\right)  {\lambda}^{-1/\alpha}(1\vee h^{-1}(\theta)^2)^{-1}\le  h^{-1}\left(\lambda \theta\right)\le  C_2^{1/\beta}h^{-1}\left({\theta}\right)  {\lambda}^{-1/\beta} (1\vee h^{-1}(\theta)^2),\ \ \lambda \ge 1.
\end{equation}
The following important result was essentially proved in \cite[Lemma 2.3]{GS20}. For the reader  convenience we provide its proof.
\begin{lemma}\label{K_est}
Let $c=(\frac {2}{C_1})^{2/\alpha}-1$.  For $0<r\le r_0$ we have
\begin{equation}
\label{hiKi}
h(r)\le c(1\vee r_0^2) K(r).
\end{equation}
 Moreover,
for $|\xi|\ge \xi_0>0$,
\begin{equation}
\label{hi_psi}
\frac 1{4c(1\vee\xi^{-2}_0)} h(1/|\xi|) \le {\psi}(\xi)\le 2h(1/|\xi|).
\end{equation}
\end{lemma}
\begin{proof} Let $r\le 1$ and $\lambda_0= \left(\frac{C_1}2\right)^{1/\alpha}<1$. Then, by (\ref{h-scaling_u0}),  we have
$$2 h(r)\le h(\lambda_0r).$$
Next, by (\ref{h_K}) and monotonicity of   $w^2K(w)$ we obtain
$$h(r)\le  h(\lambda_0r) -  h( r)= 2\int_{\lambda_0r}^r   w^2K(w)w^{-3} dw \le 2 r^2K(r)\int_{\lambda_0r}^r   w^{-3} dw $$

Hence

$$h(r) \le \frac {1-\lambda^2_0}{\lambda^2_0} K(r)=   \left(\left(\frac {2}{C_1}\right)^{2/\alpha}-1\right) K(r), r\le 1. $$

If $r_0>1$, then for $1\le r\le r_0$ then, by  monotonicity of   $w^2K(w)$ and $w^2h(w)$, we have
$$r^2h(r) \le  r_0^2h(r_0)\frac {r^2K(r)} {K(1)}.$$
Hence,
$$h(r) \le  r_0^2\frac {h(1)} {K(1)}K(r)\le \left(\left(\frac {2}{C_1}\right)^{2/\alpha}-1\right)r_0^2 K(r).$$
The proof of (\ref{hiKi})  is completed.

Next, by the inequality $1-\cos x \ge x^2/4$ for $|x|\le 1$, we obtain
 $$(1/4) K(1/|\xi|)\le {\psi}(\xi) = \int_{\R} (1 - \cos(\xi x)) d {\nu}(x)\le 2h(1/|\xi|).$$
Applying (\ref{hiKi}) we have
$$\psi(\xi)\ge \frac1{4c(1\vee\xi^{-2}_0)}  h(1/|\xi|), |\xi|\ge \xi_0,$$
which completes the proof.
\end{proof}

Now, we can give the arguments that \eqref{h-scaling_u0}  is equivalent to  \eqref{psi-scaling}. If \eqref{h-scaling_u0} holds, then \eqref{hi_psi} with $\xi_0=1$ implies \eqref{psi-scaling} with $C_1^*=\frac{C_1}{8c},
C_2^*= 8cC_2$, where $c=(2/C_1)^{2/\alpha}-1$.

On the other hand \eqref{psi-scaling} implies the same scaling conditions for the maximal function
$\psi^*(\xi)= \sup_{|x|\le |\xi|}\psi^*(\xi)$. Then \eqref{h-scaling_u0} holds, since due to  \cite[Lemma 4]{G14} we have  $\psi^*(\xi)\asymp h\left(1/|\xi|\right), \xi\in \R$.

\begin{lemma}
\label{scal_h}
Let $\tau>0$. For $ 0<u\le \tau$ we have
\begin{equation}
\label{h-1}
c_1 u^{1/\alpha} \le h^{-1}(1/u) \le c_2 u^{1/\beta},
\end{equation}
where $c_1= C_1^{1/\alpha}(h(1)\wedge \frac1\tau)^{1/\alpha}$ and $c_2= C_2^{1/\beta}\left( h^{-1}\left(\frac1\tau\right) \vee 1\right)h\left({1}\right)^{1/\beta} $.
\end{lemma}

\begin{proof}

Taking $\theta=1$ we can rewrite  (\ref{h-scaling_u0}) as
$$
C_1^{1/\alpha} h\left({1}\right)^{1/\alpha} h\left({\lambda}\right)^{-1/\alpha} \le \lambda  \le
C_2^{1/\beta} h\left({1}\right)^{1/\beta} h\left({\lambda}\right)^{-1/\beta}, \ 0<\lambda \le 1.$$
Putting  $\lambda= h^{-1}(s)$, for $s\ge h(1)$,
we have
$\displaystyle{(C_1 h\left({1}\right))^{1/\alpha} s^{-1/\alpha} \le h^{-1}(s)\le (C_2 h\left({1}\right))^{1/\beta} s^{-1/\beta}}$.

If  $0<s_0\le s\le h(1)$ we have
$\displaystyle{s_0^{1/\alpha}s^{-1/\alpha} \le h^{-1}(s) \le h^{-1}(s_0)    h\left({1}\right)^{1/\beta} s^{-1/\beta}}$.
Choosing $\frac1s =u\le \tau $ we show that
$$
C_1^{1/\alpha}\left(\frac1\tau \wedge h(1)\right)^{1/\alpha} t^{1/\alpha} \le h^{-1}(1/t) \le C_2^{1/\beta} h^{-1}\left(\frac1\tau \wedge h(1)\right)h\left({1}\right)^{1/\beta}  t^{1/\beta}.
$$
\end{proof}

Now we state an easy corollary to (\ref{h-scaling_11}) and Lemma
\ref{scal_h}.
\begin{corollary}  \label{h_inv0}
 Let  $0<\eps<1$
and    $0<u\le \tau<\infty$. We have
\begin{equation}\label{h_inv}
   \frac {h^{-1}(u^{\eps-1})}{ h^{-1}(u^{-1})}\le c u^{-\eps/\alpha},
\end{equation}
where $c=c(h(1),h^{-1}(1/\tau), h^{-1}(1), \tau,\alpha, \eps, C_1 )$.

If $\tau_0= (h(1)\vee 1)^{-\frac1{1-\eps}}$, then for $u\le \tau_0$,
\begin{equation}\label{h_inv2}
   \frac {h^{-1}(u^{\eps-1})}{ h^{-1}(u^{-1})}\le C_1^{-1/\alpha} u^{-\eps/\alpha}.
\end{equation}

Moreover, for $0< \lambda\le 1$,
\begin{equation}\label{h_inv1}
   \frac {h^{-1}( 1/u)}{ h^{-1}(1/(\lambda u))}\le  C_1^{-1/\alpha} \lambda^{-1/\alpha}  (h^{-1}(\tau^{-1})\vee1)^2.
\end{equation}
\end{corollary}

\begin{proof}


Let $u\le \tau_0= (h(1)\vee1)^{-1/(1-\eps)}$.  We apply (\ref{h-scaling_11}) with $\lambda= u^{-\eps}$ and
$\theta= u^{-1+\eps}$ to get
\begin{equation*}\label{h_inv21}
   \frac {h^{-1}(u^{\eps-1})}{ h^{-1}(u^{-1})}= \frac {h^{-1}(u^{\eps-1})}
	{ h^{-1}(u^{-\eps}u^{\eps-1})}
	\le C_1^{-1/\alpha} u^{-\eps/\alpha}.
\end{equation*}
If $\tau_0 \le u\le  \tau$ then by monotonicity of   $h^{-1}$,
$$h^{-1}(u^{\eps-1})\le  h^{-1}(1/\tau)\vee h^{-1}(1)\ \text{and } h^{-1}(u^{-1})\ge h^{-1}(\tau_0^{-1}).$$
 Moreover, by Lemma \ref{scal_h}, we obtain

$$h^{-1}(\tau_0^{-1})\ge C_1^{1/\alpha}\left(\frac1\tau \wedge h(1)\right)^{1/\alpha} \tau_0^{1/\alpha}. $$

It follows that for  $\tau_0 \le u\le  \tau$,

$$\frac {h^{-1}(u^{\eps-1})}{ h^{-1}(u^{-1})}\le \frac{h^{-1}(1/\tau)\vee h^{-1}(1)}{C_1^{1/\alpha}\left(\frac1\tau \wedge h(1)\right)^{1/\alpha}}\tau_0^{1/\alpha} \tau^{\eps/\alpha}u^{-\eps/\alpha}.$$

\end{proof}

Now we start to construct a truncated version $g$ of the transtion density $\tilde{g}$. Fix $\eps \in (0,\eps_0]$, where $\eps_0<1$ is defined in Remark \ref{remark_epsilon}.
  Let for $u>0$  $$R_u=R_u(\eps)= h^{-1}\left(\frac 1{u^{1-\eps}}\right)$$
 and $$
\psi_u(\xi) = \int_{|v|\le R_u } (1-\cos(v\xi)) \nu(dv).
$$
We have
\begin{equation}\label{H2}\psi_u(\xi)\ge \psi(\xi)- 2u^{ \eps-1}.\end{equation}
To prove (\ref{H2}) we note that
$$
 \int_{|v|> R_u } (1-\cos(v\xi)) \nu(dv)\le 2 \int_{|v|> R_u }  \nu(dv)\le 2h(R_u)= 2u^{\eps-1},
$$
hence
$$
\psi_u(\xi)= \psi(\xi)
-\int_{|v|> R_u } (1-\cos(v\xi)) \nu(dv)
\ge
 \psi(\xi)- 2u^{ \eps-1}.
$$

Let $0 < u <\infty$ and $w \in \R$. Put
$$
g_{u}(w) = \frac1{2\pi}\int_{\R} e^{iwz} e^{-\int_0^u \psi_r(z)  \, dr} \, dz.
$$
Since, by  \eqref{H2} and then  by \eqref{psi-scaling}, $\int_0^u \psi_r(z) \, dr \ge u\psi(z)- (2/\eps) u^{ \eps}\ge c |z|^\alpha$ for $|z|$ large enough, the function $g_{u}$ is well defined density function such that
$g_{u}\in C^\infty(\R)$.

For $0 < u <\infty$ and a measurable set $D \subset \R$ put
$$
\nu_u(D) = \int_0^u\int_{|x| \le R_r} 1_{D}(x)\nu(dx)  \, dr.
$$
It is clear that $
\nu_u$ is the L\'evy measure of the infinitely divisible density $g_{u}$.

For any $u > 0$ put
$$
m_u =  \int_{\R} x^2 \,  \nu_u(dx).$$
\begin{lemma} \label{2moment}
For any   $\tau>0$  there is a constant $c=c(\alpha, C_1, h^{-1}(1/\tau), h^{-1}(1))$
such that for  $u\le \tau$ we have
$$  c R_u^2 u^{ \eps }\le  m_u \le R_u^2 u^{ \eps } .$$
  The upper bound holds for any $u>0$.
	
	If $u\le \tau_0= h(1)^{-\frac1 {1-\eps}}$, then for  $u\le \tau_0$
	$$m_u \ge c  R_u^2 u^{ \eps }$$
	with $c=c(C_1, \alpha)$.
	\end{lemma}
\begin{proof}
The upper bound is clear:
$$
m_u =  \int_{\R} x^2   \nu_u(d x)\le u  \int_{|x| \le R_u} x^2  \nu(d x)\le u R_u^2 h(R_u)= R_u^2 u^\eps ,
$$
The lower bound follows from the scaling property. By (\ref{h_inv1}),
$$\frac {R_u }{  R_{u/2}}\le  \left(\frac {2^{1-\eps}} {C_1}\right)^{1/\alpha} \left(1\vee h^{-1}\left(\frac 1{\tau^{1-\eps}}\right)\right)^2.$$
By Lemma \ref{K_est}, we have for $r\le R_{\tau}$,
$$\frac {h(r) }{  K(r)}\le  \left(\frac {2} {C_1}\right)^{2/\alpha} \left(1\vee h^{-1}\left(\frac 1{\tau^{1-\eps}}\right)\right)^2.$$
Observing that $ h^{-1}(\frac 1{\tau^{1-\eps}})\le h^{-1}(1)\vee h^{-1}(1/\tau))$
and taking \newline $c= \left(\frac {2} {C_1}\right)^{-2/\alpha} (1\vee h^{-1}(1)\vee h^{-1}(1/\tau))^{-2}$
   we have
	$c h(r)\le  K(r)$ for $r\le R_\tau$. Moreover  $cR_u \le  R_{u/2}$ for $u\le \tau$.
   Then  using monotonicity of $h$ we have
\begin{eqnarray*}
m_u &=&   \int_{0}^{u}  \int_{|x| \le R_r} x^2 \nu(d x)\, dr= \int_{0}^{u}   R_r^2 K(R_r)\, dr\\
&\ge& c\int_{u/2}^{u}   R_r^2 h(R_r)dr\ge \frac c2 u R_{u/2}^2  h(R_u)   \ge \frac {c^2}2 u  R_{u}^2  h(R_u) = \frac {c^2}2  R_u^2 u^\eps.
\end{eqnarray*}
\end{proof}

\begin{lemma} \label{2moment0}
Let $\gamma>\beta  $. For any   $\tau>0$  there is a constant $c=c(C_2, \beta,\gamma, h^{-1}(1/\tau),h^{-1}(1))$
such that for  $u\le \tau$ we have
$$\int_ {|x|\le R_u} |x|^\gamma \nu(dx) \le c R_u^\gamma h( R_u)=  c R_u^\gamma u^{-1+\eps}.$$
Moreover for  $u\le \tau_0= h(1)^{-\frac1 {1-\eps}}$ the above  constant $c=c(C_2, \beta,\gamma)$.

If  $\gamma=2$ the above inequality holds for any $u>0$ and with $c =1$.
	\end{lemma}
\begin{proof}
Let $L(r)=  \nu([r, \infty)), r>0$. By integration by parts

$$\int_ {0< x\le R_u} x^\gamma \nu(dx) \le   \limsup_{r\to 0^+}r^\gamma L(r)+
\gamma\int_ {0< x\le R_u} x^{\gamma-1} L(x)dx. $$
Next, by (\ref{h-scaling_u0}),
$$\limsup_{r\to 0^+}r^\gamma L(r)\le \limsup_{r\to 0^+}r^\gamma h(r)=0,$$
which implies
$$\int_ {0< x\le R_u} x^\gamma \nu(dx)\le \gamma \int_ {0< x\le R_u} x^{\gamma-1} L(x)dx
\le \gamma \int_ {0< x\le R_u} x^{\gamma-1} h(x)dx. $$

 It follows from (\ref{h-scaling_u0}) that, if $R_u\le 1$, then
$$h(x)= h(R_u(x/R_u))\le C_2(x/R_u)^{-\beta}h(R_u), \ x\le R_u.$$
If $1\le R_u\le R_{\tau}$, since $r^2h(r)$ is an increasing function and by (\ref{h-scaling_u0}), we have
$$h(x)\le  h(x/R_u)\le C_2(x/R_u)^{-\beta}h(R_u)R^2_u,\ x\le R_u.$$
The last two estimates yield
$$\int_ {0< x\le R_u} x^{\gamma-1} h(x)dx\le C_2(R^2_{\tau}\vee1) R_u^{\beta}h(R_u)\int_ {0< x\le R_u} x^{\gamma-\beta-1} dx= C_2(R^2_{\tau}\vee1)\frac 1{\gamma-\beta}R_u^{\gamma}h(R_u).$$ This together with the estimate
$$R_{\tau}=  h^{-1}\left(\frac1{\tau^{1-\eps}}\right)\le h^{-1}(1/\tau)\vee h^{-1}(1)$$
 end the proof for arbitrary $\gamma>\beta$. Moreover, for $\tau=\tau_0$, we have $R_{\tau}=1$ which shows that the constant $c=  C_2\frac {\gamma}{\gamma-\beta}$.

The assertion of the lemma for $\gamma=2$ is a consequence of the definition of the function $h$.
\end{proof}

\begin{lemma}\label{moment1} Let $\gamma\ge 0$ and $\tau >0$. There are  constants $c_1=c_1(\alpha, h^{-1}(1/\tau), \tau,\gamma, \eps, C_1 )$ and $c_2=c_2(\gamma)$ such that
 for any $0<u \le \tau$,
\begin{equation}\label{moments_g}
\frac {c_2}{ h^{-1}(1/u)^{\gamma+1}} \le \int_{\R}|z|^{\gamma} e^{-\int_0^u  \psi_r(z) dr} dz
\le   \frac {c_1}{ h^{-1}(1/u)^{\gamma+1}}.
\end{equation}
If $\tau = 1/h(1)$ then the constant $c_1=c_1(h(1), \eps, \alpha, \gamma,  C_1)$.
\end{lemma}
\begin{proof}
By (\ref{H2}),
$$
\int_0^u  \psi_r(z) dr\ge u\psi(z) - (2/\eps) u^{\epsilon}.
$$
    Next,
$$
\int_{\R}|z|^{\gamma} e^{-\int_0^u  \psi_r(z) dr} dz
\le e^{(2/\eps) u^{ \eps}} \int_{\R} |z|^{\gamma} e^{-u\psi(z)}dz.
$$
Moreover,  by Lemma \ref{K_est},
$$\int_{\R} |z|^{\gamma} e^{-u\psi(z)}dz\le 2 + \int_{|z|>1} |z|^{\gamma} e^{-c u  h(1/|z|)}dz,$$
 where  $c=\frac1{4\left((\frac {2}{C_1})^{2/\alpha}-1\right)}$.
By the same arguments as in the proof of \cite[Lemma 16]{BGR14} we get
$$\int_{\R} |z|^{\gamma} e^{-u c h(1/|z|)}dz \le \frac {c_*}{ h^{-1}(1/u)^{\gamma+1}},\ u\le 1/h(1),$$
 where $c_*=c_*(\gamma, \alpha, C_1)$. Since $h^{-1}(1/u)\le 1$ for $u\le 1/h(1)$, it follows that
$$\int_{\R} |z|^{\gamma} e^{-u\psi(z)}dz \le \frac {c^*}{ h^{-1}(1/u)^{\gamma+1}},\ u\le 1/h(1),$$
where $c^*=c_*+2$. If  $1/h(1) \le u\le \tau$, then from the above estimate and monotonicity of $h^{-1}$ we obtain
$$\int_{\R} |z|^{\gamma} e^{-u\psi(z)}dz \le  \frac {c^*h^{-1}(1/\tau)^{\gamma+1}}{ h^{-1}(1/u)^{\gamma+1}}.$$

Finally we conclude that
$$\int_{\R} |z|^{\gamma} e^{-u\psi(z)}dz \le c^*(h^{-1}(1/\tau)\vee1)^{\gamma+1}\frac 1{ h^{-1}(1/u)^{\gamma+1}}$$
for $u\le \tau$. The proof of the upper bound is completed.

To get the lower bound we observe that
$$
\int_0^u  \psi_r(z) dr\le u\psi(z) \le 2 u h(1/|z|).
$$
Hence, denoting $a= \frac1 {h^{-1}(1/u)}$ we arrive at
$$
\int_{\R}|z|^{\gamma} e^{-\int_0^u  \psi_r(z) dr} dz
\ge  \int_{-a}^{a} |z|^{\gamma} e^{-u2h(1/|z|)}dz\ge 2e^{-2}\int_{0}^{a} |z|^{\gamma} dz=2e^{-2}\frac1{\gamma+1}a^{\gamma+1},
$$
which ends the proof of the lower bound.
\end{proof}

\begin{corollary}  \label{gu(0)estimate}For   $0<u\le \tau$

\begin{equation}\label{g(0)estimate}
\frac {c_2}{2\pi h^{-1}(1/u)} \le g_u(0)
\le   \frac {c_1}{2\pi h^{-1}(1/u)},
\end{equation}
where $c_1, c_2$ are constants from (\ref{moments_g}) corresponding to $\gamma=0$.

The function $(0, \infty)\times\R \ni (u, x) \to g_u(x)$ is continuous.
\end{corollary}
\begin {proof} Since
$$
g_{u}(x) = \frac1{2\pi}\int_{\R} e^{ixz} e^{-\int_0^u \psi_r(z)  \, dr} \, dz.
$$
we get,  by Lemma \ref{moment1}, the lower and upper estimate of $g_{u}(0)$. The continuity follows from the continuity of the map $(0, \infty)\times\R \ni (u, x) \to e^{ixz} e^{-\int_0^u \psi_r(z)  \, dr}$, the upper estimate in (\ref{moments_g}) and the bounded convergence theorem.
\end{proof}

\begin{lemma}
\label{gu_difference}
 For any $u >0$ we have
\begin{equation}
\label{gu_difference_max}
\sup_{x \in \R} |g_u(x)- \tilde{g}_u(x)| \le 2{\frac{u^{\eps}}{\eps} } g_u(0),
\end{equation}
\begin{equation}
\label{gu_difference_integral}
\int_{\R}|g_u(x)- \tilde{g}_u(x)| \, dx \le 2{\frac{u^{\eps}}{\eps} }
\end{equation}
and
\begin{equation}
\label{gu_comp}
\tilde{g}_u(0)\le g_u(0)   \le   \tilde{g}_u(0) e^{\frac{u^{\eps}}{\eps} }.
\end{equation}

\end{lemma}
\begin{proof}
The proof is similar to the proof of Proposition C.9 in \cite{KKS2020}. Let $u \in (0,\tau]$ and $x \in \R$ be arbitrary. For any $z \in \R$ we have
$$
u \psi(z) = \int_0^u \psi_r(z) \, dr +
\int_0^u \int_{|v| > R_r} (1 - \cos(vz)) \, \nu(dv) \, dr.
$$
It follows that
\begin{eqnarray}
\nonumber
\tilde{g}_u(x) &=&
\frac{1}{2 \pi} \int_{\R} e^{ixz} e^{-u \psi(z)} \, dz\\
\nonumber
&=& \frac{1}{2 \pi} \int_{\R} e^{ixz} e^{-\int_0^u \psi_r(z) \, dr}
e^{- \int_0^u \int_{|v| > R_r} (1 - \cos(vz)) \, \nu(dv) \, dr} \, dz\\
\label{convolution}
&=& \int_{\R} g_u(x - z) P_u^{\text{tail}}(dz),
\end{eqnarray}
where $P_u^{\text{tail}}(dz)$ is the exponential (for the convolution) of the measure $\Lambda_u^{\text{tail}}$ i.e.
\begin{equation}
\label{exp_convolution}
P_u^{\text{tail}}(A) = e^{-\Lambda_u^{\text{tail}}(\R)} \sum_{k = 0}^{\infty} \frac{1}{k !} \left(\Lambda_u^{\text{tail}}\right)^{*k}(A), \quad A \in \calB(\R),
\end{equation}
where $\Lambda_u^{\text{tail}}(A) = \int_0^u \nu(\{v \in A: \, |v| > R_r\}) \, dr$.
We have
\begin{equation}
\Lambda_u^{\text{tail}}(\R) =
\int_0^u \nu(\{v \in \R: \, |v| > R_r\}) \, dr
\le \int_0^u h(R_r) \, dr
\label{lambda_u}
= \frac{u^{\eps}}{\eps}.
\end{equation}
It follows that
\begin{equation}
\label{exp_lambda_u}
\left|1 - e^{- \Lambda_u^{\text{tail}}(\R)}\right| \le \frac{u^{\eps}}{\eps}.
\end{equation}
Moreover, by (\ref{convolution}) and (\ref{exp_convolution}), we get
\begin{eqnarray}
\nonumber
|\tilde{g}_u(x) - g_u(x)| &\le&
g_u(x) \left|1 - e^{- \Lambda_u^{\text{tail}}(\R)}\right|\\
\label{gu_difference_series}
&+& e^{- \Lambda_u^{\text{tail}}(\R)} \sum_{k = 1}^{\infty} \int_{\R} \frac{1}{k !} g_u(x - z) \left(\Lambda_u^{\text{tail}}\right)^{*k}(dz).
\end{eqnarray}
Using this, (\ref{lambda_u}) and (\ref{exp_lambda_u}) we get (\ref{gu_difference_max}). Integrating (\ref{gu_difference_series}) and using (\ref{lambda_u}), (\ref{exp_lambda_u}) we get (\ref{gu_difference_integral}).

Applying (\ref{convolution}) with $x=0$  we obtain
$g_u(0) e^{- \Lambda_u^{\text{tail}}(\R)}  \le \tilde{g}_u(0) \le  g_u(0)$,
which combined with (\ref{lambda_u}) proves (\ref{gu_comp}).
\end{proof}

Let for $u > 0$, $\xi, w \in \R$
$$
v_u(\xi,w) = -\xi w + \int_{\R} (\cosh(\xi v) - 1) \, \nu_{u}(dv).
$$

\begin{lemma}
\label{theta0}
Fix $u > 0$, $w \in \R$. Let $\xi_0 \in \R$ be such that
$$
v_u(\xi_0,w) = \inf_{\xi \in \R} v_{u}(\xi,w).
$$
Then,
\begin{equation}
\label{xi_0_inequality}
 |\xi_0| \le 2 \frac {|w|}{m_u}.
\end{equation}

\end{lemma}
\begin{proof}
We have

\begin{eqnarray*}
\int_{\R} (\cosh(\xi z) - 1) \, \nu_{u}(dz)
& \ge&  \frac{1}{2}  \int_{\R} \xi^2 z^2 \, \nu_{u}(dz)
 =\frac{1}{2} \xi^2 m_u\\
\end{eqnarray*}
Hence
$
v_u(\xi,w) \ge -\xi w + \xi^2m_u/2
$.
Since  $v_u(\xi_0,w) \le v_u(0,w)= 0$ we have
$
-|\xi_0| |w| + \xi_0^2m_u/2 \le 0
$,
which gives (\ref{xi_0_inequality}).
\end{proof}

\begin{lemma} \label{exp_bound_derivative} Let $\tau>0$.
For any $0<u \le \tau $, $w \in \R$, $k \in \N_0$ we have

\begin{equation} \label{der_est_3}
\left|\frac {d^k}{d w^k} g_{u}(w) \right|
\le c_k \left(\frac1{h^{-1}(1/u)} \right)^k
g_{u}(0)
\end{equation}
and
\begin{equation} \label{der_est_1}
\left|\frac {d^k}{d w^k} g_{u}(w) \right|
\le c_k \left(\frac1{u^\eps h^{-1}(1/u)} \right)^k e^{-{\frac{|w|}{8 R_u}}}
g_{u}(0) .
\end{equation}
The constant $c_k$ depends on $C_1, \alpha, \eps, \tau, h^{-1}(1/\tau), h^{-1}(1)$ and $k$.
If \newline $\tau=\tau_0= (h(1)\vee 1)^{-\frac1{1-\eps}}$, then the constant $c_k$ depends on $k, C_1, \alpha, \eps, h(1)$.

\end{lemma}
\begin{proof}
The proof of (\ref{der_est_3}) follows immediately from Lemma \ref{moment1} and Corollary  \ref{gu(0)estimate}.

Let
$$Q_{u}(\xi,w)= i \xi w - \int_0^u \psi_{r}(\xi) \, dr,\  u>0,\ \xi, w \in \R. $$
We have
$$
g_{u}(w) 
= \frac{1}{2 \pi} \int_{\R} e^{Q_u(\xi,w)} \, d \xi.
$$
For any $k \in \N_0$ we get
$$\frac{d^k}{d w^k} g_{u}(w)
= \frac{1}{2 \pi} \int_{\R} i^k \xi^k   e^{Q_u(\xi,w)} d \xi$$

Recall that  $\xi_0= \text {arg min}_\xi v_u(\xi, w)$. We proceed in a similar way as in \cite{KS2012} where a bound on the transition density was derived. Note that functions $\psi_r$ and $Q_u(\cdot,w)$ can be extended analytically to $\C$.
 Applying the Cauchy-Poincare theorem (justification is exactly the same as in the proof of Theorem 6 of \cite{KS2012})
we claim that
$$\left|\frac{d^k}{d w^k} g_{u}(w) \right|=
\frac{1}{2 \pi} \left|\int_{\R} (\xi+i\xi_0)^k e^{Q_u(\xi+i\xi_0,w)}d \xi\right|.$$

Observe that
$$
\text{Re}\, Q_{u}(\xi+i\xi_0,w)
\le v_u(\xi_0,w)- \int_0^u \psi_r(\xi) \, dr.
$$

Hence
\begin{eqnarray*}\left|\frac{d^k}{d w^k} g_{u}(w)\right|
&\le& \int_{\R} (|\xi|+|\xi_0|)^k
e^{v_u(\xi_0,w)- \int_0^u \psi_r(\xi) \, dr} d \xi\\
&\le& 2^k e^{v_u(\xi_0,w)}
\int_{\R} (|\xi|^k+|\xi_0|^k)
e^{- \int_0^u \psi_r(\xi) \, dr} d \xi
\end{eqnarray*}

Now, we will show that for any $w \in \R$ we have
\begin{eqnarray}
{v_u(\xi_0,w)}  &\le&
{\frac{ em_u}{2R^2_u}} {-\frac{|w|}{4 R_u}} \label{v_u_bound}\\
&\le& \frac{ e}2u^\epsilon {-\frac{|w|}{4 R_u}}.\label{v_u_bound1}
\end{eqnarray}
If $|w| \le \frac{2 em_u}{R_u}$, then
$${\frac{ em_u}{2R^2_u}} {-\frac{|w|}{4 R_u}}\ge 0= v_u(0,w)\ge v_u(\xi_0,w),$$
which proves (\ref{v_u_bound}) in this case.

If $|w| \ge \frac{2 em_u}{R_u}$,
to prove (\ref{v_u_bound}),  we use the arguments as in \cite[proof of Lemma 4.2]{S2017}, so we omit the details.

Next, observe that (\ref{v_u_bound1}) follows from Lemma \ref{2moment}.


We also have, by (\ref{moments_g}),
$$
\int_{\R} |\xi|^k e^{- \int_0^u \psi_r(\xi) \, dr} d \xi
\le \frac {c_k}{ h^{-1}(1/u)^{k+1}}.
$$
By Lemma \ref{theta0},
$$
\int_{\R} |\xi_0|^k
e^{- \int_0^u \psi_r(\xi) \, dr} d \xi
\le c_k |w|^k  \frac c{ m^k_uh^{-1}(1/u)}.
$$
Hence
\begin{eqnarray*}
\left|\frac {d^k}{d w^k} g_{u}(w) \right|
&\le& \frac {c_k}{ h^{-1}(1/u)^{k+1}}\left(1+ \frac {R_u h^{-1}(1/u)}{m_u} \frac {|w|}{R_u}\right)^k e^{-{\frac{|w|}{4 R_u}}}\\
&\le& \frac {c_k}{ h^{-1}(1/u)^{k+1}}\left(1+ \frac {R_u h^{-1}(1/u)}{m_u}\right)^k  e^{-{\frac{|w|}{8 R_u}}}.
\end{eqnarray*}
Next, we observe that, by Lemma  \ref{2moment} and  since $h^{-1}(1/u)\le \left(\frac{h^{-1}(1/\tau)}{h^{-1}(1)}\vee1\right) R_{u}$, we obtain

$$\frac {R_u h^{-1}(1/u)}{m_u}\le  c\frac { h^{-1}(1/u)}{u^\eps R_u}\le  \frac c {u^\eps}  $$
which implies

$$\left|\frac {d^k}{d w^k} g_{u}(w) \right|
\le \frac {c_k}{u^{k\eps} h^{-1}(1/u)^{k+1}}  e^{-{\frac{|w|}{8 R_u}}}.$$

Finally, we note that all the constants in the case $\tau=\tau_0= (h(1)\vee 1)^{-\frac1{1-\eps}}$ are dependent only on $k, C_1, \alpha, \eps,h(1)$. It follows from the appropriate parts of  Lemma  \ref{2moment} and Lemma \ref{moment1}.
\end{proof}

Let  for $0<u<\infty$ and  $f\in C^2(\R)$
$$
 K_u f(w) =  \mathrm{P.V.} \int_{|z| \le R_u} (f(w + z) - f(w)) \, \nu(dz), \ w\in \R.
$$
For $f \in L^1(\R)$ and $\xi \in \R$ denote $\hat{f}(\xi) = \int_{\R} e^{-i\xi x} f(x) \, dx$.
If $f\in C^2(\R)\cap L^1(\R) $ is such that $\limsup_{|x|\to \infty}|x|^{1+\delta}|f''(x)|<\infty$ for some
 $\delta>0$ we observe that $K_u f \in  L^1(\R)$ and
\begin{equation}
\label{fourier}
\widehat{K_u f}(\xi)= -\psi_u(\xi) \hat{f}(\xi),\ \xi\in \R.
\end{equation}
Next, by \eqref{der_est_1}, the above requirements are satisfied for $f=g_u$.

 We have $\hat{g}_u(\xi) = e^{-\int_0^u \psi_r(\xi) \, dr}$, so for any $0 < u <\infty $ and $\xi \in \R$
$$
\frac{\partial}{\partial u} \hat{g}_u(\xi) + \psi_u(\xi) \hat{g}_u(\xi) = 0.
$$
By \eqref{fourier}, it follows that for any $0 < u <\infty$ and $w \in \R$ we have
\begin{equation}
\label{1dim_equation1}
\frac{\partial}{\partial u} \hat{g}_u(w) - \widehat {K_u g_u}(w) = 0.
\end{equation}

Next we claim  that
$$\frac{\partial}{\partial u} g_u(w)= \frac1{2\pi}\int_{\R} e^{iwz} \frac{\partial}{\partial u}\hat{g_u}(z) \, dz.$$
This follows from the estimate
$$|\frac{\partial}{\partial u}\hat{g_u}(z)|= \psi_u(z) \hat{g}_u(z) \le e^{(2/\eps) u^{ \eps}} \psi(z)e^{-u\psi(z)},$$
which is implied by  \eqref{H2}.
Next we use \eqref{1dim_equation1}
 to get
\begin{equation} \label{1dim_equation}\frac{\partial}{\partial u} g_u(w)= \frac1{2\pi}\int_{\R} e^{iwz} \widehat{K_u g_u}(z) \, dz= {K_u g_u}(w),  \end{equation}
where we have equality almost surely. By continuity we have it everywhere.

\section{Parametrix construction}\label{Parametrix construction}

This highly technical section contains detailed proofs of a number of  facts and estimates needed to provide the construction of the fundamental solution  $p_{t,s}(x,y)$, which was explained earlier in Section \ref{s3}. The construction demands many auxiliary results, in particular key  estimates of the zero-order approximation term $p_{t,s}^{(0)}(x,y)$ and the kernel $q_{t,s}^{(0)}(x,y)$ contained in Lemma \ref{py_integrable} and Lemma  \ref{q0_estimates_lemma}, respectively.

In this section we adopt the convention that constants denoted by $c$ (or $c_1, c_2, \ldots$) may change their value from one use to the next. Unless is explicitly stated otherwise, we understand that constants denoted by $c$ (or $c_1, c_2, \ldots$) depend only on $d$, $\alpha$, $\beta$, $\gamma_1$, $\gamma_2$, $\gamma_3$, $C_1, \ldots, C_7$. We  also understand that they may depend on on the choice of the constant $\eps$. We write $c = c(a,b, \ldots)$ when $c$ depends on the above constants and additionally on $a$, $b$, \ldots. For a square matrix $A$
we denote by $|A|$ its standard operator norm. The standard inner product for $x, y \in \R^d$ we denote by $x y$.
\begin{remark}
\label{remark_epsilon}
Our choice of  $p_{t,s}^{(0)}(x,y)$ and the kernel $q_{t,s}^{(0)}(x,y)$ will depend on the given value of $\eps>0$. In order to have required  bounds  involving  these objects, we have to impose a restriction on $\eps$. Namely throughout the rest of the paper we assume that $\eps\le \eps_0$, with $\eps_0$ defined below.

 In the case (A) we set
\begin{eqnarray*}\eps_0&=& \min \left\{ \frac{\gamma_1 \alpha}{2(d+3)\beta}, \frac{\gamma_2 \alpha}{2(d+3)},
\frac{  \gamma_1/(\beta(1+\gamma_1))}{2+2/\alpha +  \gamma_1/(\beta(1+\gamma_1))  },\frac {\gamma_3-1}{\gamma_3-1+\beta(1+(d+1)/\alpha)}\right\}, \end{eqnarray*}
while in the case (B) we pick
\begin{eqnarray*}\eps_0&=& \min \left\{ \frac{(1 + \gamma_1)/\beta - 1/\alpha}{2(d+3)/\alpha}, \frac{\gamma_2 - \left(1/\alpha - 1/\beta\right)}{2(d+3)/\alpha},
\frac{ 1/\beta - 1/((1+\gamma_1)\alpha)}{2+2/\alpha +1/\beta - 1/((1+\gamma_1)\alpha)},\right.\\
&& \quad \quad \quad \left.
\frac {\gamma_3-\beta/\alpha}{\gamma_3+\beta(1+d/\alpha)}\right\}. \end{eqnarray*}
Due to our assumptions $\eps_0$ is positive.
\end{remark}

 For $i \in \{1,\ldots,d\}$, $u > 0$ put
$$
R^{(i)}_u=R^{(i)}_u(\eps)= h_i^{-1}\left(\frac 1{u^{1-\eps}}\right).
$$

Let $g_{u}^{(i)}= g_{u}^{(i,\eps)}$ be the truncated density corresponding to $\psi_i$ according to the truncation procedure described in Section \ref{1dim}.

For any $u > 0$, $x \in \R^d$ define
$$
G_u(x)= g_{u}^{(1)}(x_1)g_{u}^{(2)}(x_2)\cdots g_{u}^{(d)}(x_d).
$$

Let $h^{-1}_{max}(r)= \max_j h^{-1}_{j}(r)$ and let  $h^{-1}_{min}(r)= \min_j h^{-1}_{j}(r)$.
Let $M_{u}$ be a diagonal $d$ - dimensional square matrix with the diagonal  $8(R^{(1)}_{u}, \dots R^{(d)}_{u})$. The multiplier $8$ is only for the notational convenience.
Note that     $|M_{u}|= 8\max_i R^{(i)}_u= 8h^{-1}_{max}\left({u^{-1+\eps}}\right)$ and   $|M^{-1}_{u}|= \frac18\max_i \frac1{R^{(i)}_u}= \frac18\frac1{h^{-1}_{min}\left({u^{-1+\eps}}\right)}$.

We recall that for $u>0$ we defined
$$
\kappa(u) = (u, h_1(1), \ldots, h_d(1),  h^{-1}_1(1), \ldots, h^{-1}_d(1), h_1^{-1}(1/u), \ldots, h_d^{-1}(1/u)).
$$

Observe that for every $\tau>0$
\begin{equation}\label{Mnorm} \frac1c u^{(\eps-1)/\beta} \le |M^{-1}_{u}|\le  c u^{(\eps-1)/\alpha},\ 0<u\le \tau, \end{equation}
and
\begin{equation}\label{Mnorm1} \frac1c u^{(1-\eps)/\alpha } \le |M_{u}|\le  c u^{(1-\eps)/\beta},\ 0<u\le \tau, \end{equation}
where $c$  depends also on $\tau$ trough the vector $\kappa(\tau)$, that is in our notation $c= c(\kappa(\tau))$. This follows from Lemma  \ref{scal_h}.
Throughout the  whole section we set
$$\tau_0=  \min_{k\le d}\left\{ (h_k(1)\vee 1)^{-\frac1{1-\eps}}\right\}.$$
 Then the constant $c$ in (\ref{Mnorm}) and  (\ref{Mnorm1}) depends on the vector $\bar{h}=(h_1(1), \dots, h_d(1)) $, if $\tau\le \tau_0$, in our convention we write $c=c(\bar{h})$. Again, this follows from Lemma  \ref{scal_h}.

The following lemma follows easily from Corollary \ref {gu(0)estimate} and Lemma \ref{exp_bound_derivative}.
\begin{lemma}
\label{Gest}

Fix $\tau > 0$. For any $u \in (0,\tau]$, $x \in \R^d$ we have
$$ c^{-1} \prod_{i=1}^d \frac1{h_i^{-1}(1/u)}\le G_u(0)\le c \prod_{i=1}^d \frac1{h_i^{-1}(1/u)},$$
$$G_u(x) \le  c  G_u(0)  e^ {-|xM^{-1}_{u}|}, $$
$$
\left|\frac{\partial}{\partial x_i } G_u(x)\right|
\le c G_u(0) \frac1{h_i^{-1}(1/u)}u^{-\eps} e^{- |xM^{-1}_{u}|}, $$
$$
\left|\frac{\partial^2}{\partial x_i \partial x_j} G_u(x)\right|
 \le c G_u(0) \frac1{h_i^{-1}(1/u)h_j^{-1}(1/u)}u^{-2\eps} e^{- {|xM^{-1}_{u}|}},
$$
where the constant $c$ depends additionally on $\tau$, that is $c=c(\kappa(\tau))$. If $\tau\le \tau_0$, then $c= c(\bar{h})$.
\end{lemma}
\begin{proof} The first estimate follows directly from Corollary \ref {gu(0)estimate}.

All the remaining  estimates follow directly from Lemma \ref{exp_bound_derivative} and the observation
$$\sum^d_{j=1}{\frac{|x_j|}{8 R^{(i)}_u}}\ge \left[\sum^d_{j=1}\left(\frac{|x_j|}{8 R^{(i)}_u}\right)^2\right]^{1/2}
= |xM^{-1}_{u}|. $$
\end{proof}

By (\ref{h_inv1}) and the above lemma we obtain the following corollary.
\begin{corollary}  \label{Gu(0)estimate}
 Fix $\tau > 0$. For any $u \in (0,\tau]$
$$
G_{u\lambda}(0)\leq c  \lambda^{-d/\alpha} G_{u}(0), \quad \lambda\le 1,
$$
where $c = c(\kappa(\tau))$. If $\tau\le \tau_0$, then $c= c(\bar{h})$.
\end{corollary}

For $0 \le t < s$ and $x, w, y \in \R^d$ put
\begin{equation}\label{densityrep}
p_{t,s}^x(w) = \frac1{|\det A_s(x)|}G_{s-t}(w (A_s^{-1}(x))^T).
\end{equation}
and
$$
L_{t,s}^y f(x) = \sum_{k = 1}^d \mathrm{P.V.} \int_{|u| < R^{(k)}_{s-t}} [f(x + u e_k A_s^T(y)) - f(x)] \, \nu_k (du).
$$
By (\ref{1dim_equation}), for any $0 \le t < s$ and $w, y \in \R^d$, we have
$$
\left(\frac{\partial}{\partial t} + L_{t,s}^y\right) p_{t,s}^y(w) = 0.
$$

We choose our zero-order approximation $p_{t,s}^{(0)}$ in  parametrix construction as
\begin{equation}
\label{p_t_s_0}
p_{t,s}^{(0)}(x,y) = p_{t,s}^y(x-y).
\end{equation}
Before we come to crucial estimates of $q_{t,s}^{(0)}$ (defined in Section \ref{s31}) we need to show some auxiliary results on $p_{t,s}^w$.

Let
\begin{eqnarray*}
\|A\| &=& \max\left(\sup_{t > 0, x\in \R^d}  |A_t^T(x)|, \sup_{t > 0, x \in \R^d} | A_t^{-1}(x)|, \sup_{t > 0, x, y\in \R^d, x \ne y}
\frac{| A_t^T(x)-A_t^T(y)|}{|x - y|^{\gamma_1}},\right. \\
&& \left.  \sup_{t > 0, x, y\in \R^d, x \ne y}
\frac{|(A_t^T(x))^{-1}-(A_t^T(y))^{-1}|}{|x - y|^{\gamma_1}},
\sup_{s > t > 0, x\in \R^d, }
\frac{|(A_t^T(x))^{-1}-(A_s^T(x))^{-1}|}{(s - t)^{\gamma_2}}
\right).
\end{eqnarray*}

 It is clear that $\|A\|\ge 1$. Note that $\|A\|$ may be bounded from above by a constant which depends only on
$d$, $C_3, \ldots, C_6$. Using standard calculations and the conditions (\ref{bounded}, \ref{determinant}, \ref{Holder}, \ref{Lipschitz}) we have
$$ ||A||\le (C_3 +C_5 +C_6)d + \frac{C^{d-1}_3}{C_4} d +(C_5+C_6)\frac{C^{d-1}_3}{C_4} d^2.$$

From Lemma \ref{Gest}, Lemma \ref{scal_h} and (\ref{densityrep}) we easily get the following corollary.
\begin{corollary}
\label{general_matrix}
Let $\tau>0$.  Then there is a constant $c= c(\kappa(\tau))$ such that for $0  < s-t\le \tau $ and $y,w\in \R^d$
\begin{equation}\label{densityest}
\left| p_{t,s}^y(w)\right|
 \le c  G_{s-t}(0)  e^{-|w (A_s^{-1}(y))^TM^{-1}_{s-t}|},
\end{equation}
\begin{equation}\label{densityest1}
\left| p_{t,s}^y(w)\right|
\le cG_{s-t}(0) e^{-\frac{1}{||A||} \frac{|w |}{ |M_{s-t}| }},
\end{equation}
\begin{equation} \label{gradest}
\left|\nabla p_{t,s}^y(w)\right|
\le c  |M^{-1}_{s-t}|(s-t)^{-\eps(1+ 1/\alpha)} G_{s-t}(0) e^{-|w (A_s^{-1}(y))^TM^{-1}_{s-t}|} ,
\end{equation}
\begin{equation} \label{hessianest}
\left|\nabla^2 p_{t,s}^y(w)\right|
\le c|M^{-1}_{s-t}|^2(s-t)^{-\eps(2+ 2/\alpha)} G_{s-t}(0) e^{-|w (A_s^{-1}(y))^TM^{-1}_{s-t}|},
\end{equation}
\begin{equation} \label{hessianest1}
\left|\nabla^2 p_{t,s}^y(w)\right|
\le c |M^{-1}_{s-t}|^2(s-t)^{-\eps(2+ 2/\alpha)} G_{s-t}(0) e^{-\frac{1}{||A||} \frac{|w |}{ |M_{s-t}| }}.
\end{equation}
If $\tau\le \tau_0$, then $c= c(\bar{h})$.
\end{corollary}
\begin{proof} We provide the proof only for (\ref{hessianest}) and (\ref{hessianest1}) since the other estimates can be shown in a similar fashion.
Applying Lemma \ref{Gest} and (\ref{densityrep}) we obtain
$$\left|\nabla^2 p_{t,s}^y(w)\right|  \le c G_{s-t}(0) \frac1{(h_{min}^{-1}((s-t)^{-1}))^2}(s-t)^{-2\eps} e^{- {|w (A_s^{-1}(y))^TM^{-1}_{s-t}|}}. $$
Next, by (\ref{h_inv}),

\begin{equation*}\label{Dmeasure} \frac {1} {|M^{-1}_{s-t}|h_{min}^{-1}((s-t)^{-1})}=   \frac {h_{min}^{-1}((s-t)^{\eps-1})} {8h_{min}^{-1}((s-t)^{-1})}\le
c(s-t)^{-\eps/\alpha},  \end{equation*}
which completes the proof of (\ref{hessianest}).

Since $|w (A_s^{-1}(y))^TM^{-1}_{s-t}|\ge \frac{1}{||A||} \frac{|w |}{ |M_{s-t}| }$ the estimate  (\ref{hessianest}) yields (\ref{hessianest1}).
\end{proof}

The following lemma is a simple consequence of the change of variable formula, hence its proof is omitted.
\begin{lemma} \label{int}
Let $\rho>0, x\in \R^d, 0\le t<s< \infty$.  There is a constant $c= c(\rho) $ such that 
$$\int_{\R^d} |x-y|^{\rho}e^{-|(x-y) (A_s^{-1}(x))^TM^{-1}_{s-t}|}dy  \le   c |M_{s-t}|^{\rho} det (M_{s-t}). $$
\end{lemma}

\begin{lemma}\label{point} Let $x,y\in \R^d$, $0\le t<s$ and  $\delta >0$. Let
\begin{equation*}
\xi = y-x +   \theta u e_k A_t^T(x)+ (1-\theta)u e_k A_s^T(y)) +\lambda U_t(x, e_ku),\ |u|\le R^{(k)}_{s-t},
\end{equation*}
where  $0\le \theta, \lambda\le 1$.

If  $|x-y|^{1+\gamma_1}\le  \frac{(s-t)^\delta} {|M^{-1}_{s-t}|}$ and $0< s-t\le \tau$,
then
\begin{equation}
\label{xiest0}
|\xi (A^T_s(y))^{-1} M^{-1}_{s-t}-  (y-x)(A^T_s(x))^{-1}M^{-1}_{s-t}| \le c.
\end{equation}
 The contant $c= c(\delta, \kappa(\tau))$. If $\tau\le \tau_0$, then $c= c(\delta,\bar{h})$.
\end{lemma}
\begin{proof} We denote $z=y-x$,  $ U= U_t(x, e_ku),\ R^{(k)}= R^{(k)}_{s-t}$ and $M= M_{s-t}$.
It is clear that it is enough to consider the case when $\theta=1$ or $\theta=0$. We provide the argument if $\theta=1$ since the case $\theta=0$ is similar, if not easier.
Let  $\xi = z +    u e_k A_t^T(x)$. Noting that  $u e_k M^{-1}=  (u/R^{(k)}) e_k$,
we obtain

\begin{eqnarray*}&&|\xi (A^T_s(y))^{-1} M^{-1}-  z(A^T_s(x))^{-1}M^{-1}|\\&&=
|z [(A^T_s(y))^{-1} -  (A^T_s(x))^{-1}]M^{-1} +  u e_k [A_t^T(x)-A^T_s(y)] (A^T_s(y))^{-1}M^{-1}\\
&&+u e_k M^{-1}+  U (A^T_s(y))^{-1}M^{-1}|\\
&&\le ||A|||z|^{1+\gamma_1} |M^{-1}|+||A|||u||z|^{\gamma_1}|M^{-1}|+ |u|/R^{(k)}+ C_7||A|| |M^{-1}| |u|^{\gamma_3}\\
&& \le ||A|| \left( |z|^{1+\gamma_1} |M^{-1}|+R^{(k)}|z|^{\gamma_1}|M^{-1}|+ 1+C_7 |M^{-1}| (R^{(k)})^{\gamma_3}\right)\end{eqnarray*}


In the case (A) ( when $h_1=\dots=h_d$)  we have $ R^{(k)}|M^{-1}|=1/8$, hence, by (\ref{Mnorm1}),

$$ |z|^{1+\gamma_1} |M^{-1}|+R^{(k)}|z|^{\gamma_1}|M^{-1}|+ 1+C_7 |M^{-1}| (R^{(k)})^{\gamma_3}\le (s-t)^\delta + |z|^{\gamma_1} +1+ C_7|M^{-1}|^{1-\gamma_3}\le c, $$
with $c=c(\delta, \kappa(\tau))$.

In the  case (B),   by (\ref{Mnorm}) and (\ref{Mnorm1}), \begin{eqnarray*}&&|R^{(k)}|z|^{\gamma_1}||M^{-1}|
\le c (s-t)^{(1-\eps)/\beta} \left(\frac{(s-t)^\delta} {8|M^{-1}|}\right)^{\frac{\gamma_1}{\gamma_1+1}}|M^{-1}|\\
&&=  c(s-t)^{(1-\eps)/\beta+\delta\frac{\gamma_1}{\gamma_1+1}} |M^{-1}|^{\frac{1}{\gamma_1+1}}
\\&&\le c(s-t)^{(1-\eps)/\beta+\delta\frac{\gamma_1}{\gamma_1+1}} (s-t)^{\frac{-1+\eps}{(\gamma_1+1)\alpha}}=c (s-t)^{\delta_0}, \end{eqnarray*}
where  $ \delta_0=\delta\frac{\gamma_1}{\gamma_1+1}+(1-\eps)(\frac{1}{\beta}-\frac1{(\gamma_1+1)\alpha})>0$. Moreover, again by (\ref{Mnorm}) and (\ref{Mnorm1}),
$$|M^{-1}| (R^{(k)})^{\gamma_3}\le c(s-t)^{(1-\eps)(\frac{\gamma_3}{\beta}-\frac{1}{\alpha})}. $$
This implies that
\begin{eqnarray*}
&& |z|^{1+\gamma_1} |M^{-1}|+R^{(k)}|z|^{\gamma_1}|M^{-1}|+ 1+\eta_5 |M^{-1}| (R^{(k)})^{\gamma_3}\\
&&\le c(s-t)^\delta + c(s-t)^{\delta_0} +1+ c(s-t)^{(1-\eps)(\frac{\gamma_3}{\beta}-\frac{1}{\alpha})}\le c,
\end{eqnarray*}
with $c=c(\delta, \kappa(\tau))$.
Hence, in both cases, the proof of (\ref{xiest0}) is completed.
\end{proof}

For any $x\in \R^d, \delta>0$ let
$$
D(\delta,x) = \left\{w \in \R^d: \, |x-w|^{1+\gamma_1}\le  \frac{(s-t)^\delta} {|M^{-1}_{s-t}|} \right\}.
$$

\begin{lemma}
\label{G_difference}
Fix $\tau > 0$ and $\delta > 0$. For any $x \in \R^d$, $y \in D(\delta,x)$, $0  < s-t \le \tau$ we have
\begin{eqnarray}
\nonumber
&& \left|G_{s-t}\left((x-y) (A^T_s(y))^{-1} \right) -
G_{s-t}\left((x-y) (A^T_s(x))^{-1}\right)\right|\\
\label{G_difference_inequality}
&& \le c (s - t)^{\delta - \eps(1+1/\alpha)} G_{s - t}(0)
e^{-|(x-y) (A_s^{-1}(x))^T M^{-1}_{s-t}|}
\end{eqnarray}
and 
\begin{equation}
\label{G_difference_integral}
\int_{D(\delta,x)} \left|G_{s-t}\left((x-y) (A^T_s(y))^{-1} \right) -
G_{s-t}\left((x-y) (A^T_s(x))^{-1}\right)\right| \, dy
\le c (s - t)^{\delta - (d+3) \eps/\alpha}.
\end{equation}
 The contant $c= c(\delta, \kappa(\tau))$. If $\tau\le \tau_0$, then $c= c(\delta,\bar{h})$.\end{lemma}
\begin{proof}
We have
\begin{eqnarray*}G_{s-t}\left((x-y) (A^T_s(y))^{-1} \right)&=&   G_{s-t}\left((x-y) (A^T_s(x))^{-1}\right)\\&+& \nabla G_{s-t}(\xi) \left[(x-y)((A^T_s(y))^{-1}-(A^T_s(x))^{-1})\right],\end{eqnarray*}
where $\xi=\theta (x-y)(A^T_s(x))^{-1} + (1-\theta) (x-y)(A^T_s(y))^{-1}, \ 0\le \theta\le 1$.
Next, we observe that
\begin{equation}\label{diff}|(x-y) (A^T_s(y))^{-1}- (x-y) (A^T_s(x))^{-1}|\le ||A|||x-y|^{1+\gamma_1}\le (s-t)^\delta||A||\frac1{|M^{-1}_{s-t}|}.\end{equation}
 Hence,
$$|\xi M^{-1}_{s-t}-  (x-y)(A^T_s(x))^{-1}M^{-1}_{s-t}|\le (s-t)^\delta |M^{-1}_{s-t}| ||A|||M^{-1}_{s-t}|^{-1} = (s-t)^\delta ||A||.$$
This implies, by Lemma \ref{Gest},  that
$$|\nabla G_{s-t}(\xi)|\le c  \frac1{h^{-1}_{min}(1/(s-t))}(s-t)^{-\eps} G_{s-t}(0) e^{-|(x-y) (A_s^{-1}(x))^TM^{-1}_{s-t}|}, $$
which together with (\ref{diff}) yield
\begin{eqnarray*}
&&\left|\nabla G_{s-t}(\xi) \left[(x-y)((A^T_s(y))^{-1}-(A^T_s(x))^{-1})\right]\right|\\
&&\le c(s-t)^\delta\frac1{|M^{-1}_{s-t}|}\frac1{h^{-1}_{min}(1/(s-t))}(s-t)^{-\eps} G_{s-t}(0) e^{-|(x-y) (A_s^{-1}(x))^TM^{-1}_{s-t}|}\\
&&\le c(s-t)^\delta (s-t)^{-\eps(1+1/\alpha)} G_{s-t}(0) e^{-|(x-y) (A_s^{-1}(x))^TM^{-1}_{s-t}|},\end{eqnarray*}
since, by Corollary  \ref{h_inv0},
$$\frac1{|M^{-1}_{s-t}|}\frac1{h^{-1}_{min}(1/(s-t))}= 8 \frac{h^{-1}_{min}((s-t)^{\eps-1})}{h^{-1}_{min}((s-t)^{-1})}\le  c (s-t)^{-\eps/\alpha}.$$
The proof of (\ref{G_difference_inequality}) is completed.

Applying (\ref{G_difference_inequality}) and Lemma \ref{int} we get
\begin{eqnarray*}
&& \int_{D(\delta,x)} \left|G_{s-t}\left((x-y) (A^T_s(y))^{-1} \right) -
G_{s-t}\left((x-y) (A^T_s(x))^{-1}\right)\right| \, dy\\
&& \le c\, det (M_{s-t})G_{s-t}(0)(s-t)^{\delta -\eps-\eps/\alpha}.
\end{eqnarray*}
Next, by Lemma
\ref{Gest} and Corollary  \ref{h_inv0},
\begin{equation} \label{det_G}
 det (M_{s-t})G_{s-t}(0)\le c\prod_{k=1}^{d} \frac {h_k^{-1}((s-t)^{\eps-1})} {h_k^{-1}((s-t)^{-1})}\le
c(s-t)^{-d\eps/\alpha},
\end{equation}
which implies (\ref{G_difference_integral}).
\end{proof}

\begin{lemma}
\label{py_integrable} Fix $\tau > 0$.
There exists $c(\kappa(\tau))=c > 0$ such that for any $0 <s-t \le  \tau$ and $x \in \R^d$ we have
\begin{equation}
\label{py_integral}
\int_{\R^d} p_{t,s}^y(x-y) \, dy \le c.
\end{equation}
Moreover,  for any  $0 <s-t \le  \tau$ and $x \in \R^d$ we have
\begin{equation}
\label{py_px_difference}
\sup_{y \in \R^d} \left|p_{t,s}^y(x-y) - p_{t,s}^x(x-y)\right| \le c G_{s - t}(0) (s - t)^{\eps}
\end{equation}
and
\begin{equation}
\label{py_px_difference_int}
\int_{\R^d} \left|p_{t,s}^y(x-y) - p_{t,s}^x(x-y)\right| \, dy \le c (s - t)^{\eps}.
\end{equation}
  If $\tau\le \tau_0$, then $c= c(\bar{h})$.
\end{lemma}
\begin{proof}
Let $\delta>0$ and $z$ be $x$ or $y$. For  $y \in  D^c(\delta,x)$ we use (\ref{densityest1}) and Lemma \ref{scal_h} to have
\begin{equation}
\label{exp_estimates}
p_{t,s}^z(x-y) \le c (s-t)^{-d/\alpha}\exp \left(-\frac{|x-y|}{||A|||M_{s-t}|} \right).
\end{equation}

In the case (A) we have  $|M^{-1}_{s-t}|^{-1}=|M_{s-t}|$.  Hence, by (\ref{Mnorm1}),
\begin{eqnarray*} \frac{|x-y|} {|M_{s-t}|}&\ge& \frac{\left[(s-t)^\delta  |M_{s-t}|\right]^{\frac1{1+\gamma_1}}}
{|M_{s-t}|}= \left[(s-t)^\delta  |M_{s-t}|^{-\gamma_1}\right]^{\frac {1}{1+\gamma_1}}\\
 &\ge&  c \left[(s-t)^{\delta - (1-\eps)\frac {\gamma_1}{ \beta }}\right]^\frac1{1+\gamma_1}.  \end{eqnarray*}
In the case (A) we choose $\delta= \frac {(1-2\eps)\gamma_1}{ \beta }  < \frac {(1-\eps)\gamma_1}{ \beta }$. Then clearly the exponent at $s - t$ is negative.

Next, we observe that in the case (B), by  (\ref{Mnorm}) and (\ref{Mnorm1}),
$$ \frac{|x-y|} {|M_{s-t}|}\ge \frac{\left[(s-t)^\delta |M^{-1}_{s-t}|^{-1}\right]^{\frac1{1+\gamma_1}}}
{|M_{s-t}|}\ge  c (s-t)^{\frac{\delta + (1-\eps)/\alpha - (1-\eps)(1+\gamma_1)/\beta }{1+\gamma_1}}. $$
In the case (B) we choose $\delta =  (1-2\eps)\left[  (1+\gamma_1)/\beta-1/\alpha\right]<  (1-\eps)\left[  (1+\gamma_1)/\beta-1/\alpha\right]$. Then clearly the exponent at $s - t$ is negative.

Hence, in both cases,  we find $c= c( \kappa(\tau) )$
such that for $y \in D^c(\delta,x)$
\begin{equation}
\label{max1}
p_{t,s}^z(y-x) \le c (s-t)^{-d/\alpha}  \exp \left(-\frac{|x-y|}{||A|||M_{s-t}|} \right) \le c (s-t)^\eps.
\end{equation}
Similarly,  we find $c= c( \kappa(\tau)) $ such that
\begin{equation}
\label{int1}
\int_{ D^c(\delta,x)}  p_{t,s}^z(y-x)dy \le c (s-t)^{-d/\alpha} \int_{ D^c(\delta,x)} \exp \left(-\frac{|x-y|}{||A|||M_{s-t}|} \right) \, dy \le c (s-t)^\eps.
\end{equation}

For any $x \in \R^d$,  $0 <s-t \le  \tau$ we have, by (\ref{G_difference_integral}),
\begin{eqnarray}
\int_{D(\delta,x)}p_{t,s}^y(x-y)&\le&
 \frac{1}{C_4} \int_{D(\delta,x)} G_{s-t}\left((x-y) (A^T_s(y))^{-1} \right) \, dy   \nonumber\\
& \le& \frac{1}{C_4}
\int_{D(\delta,x)} \left|G_{s-t}\left((x-y) (A^T_s(y))^{-1} \right)
- G_{s-t}\left((x-y) (A^T_s(x))^{-1} \right)\right| \, dy  \nonumber\\
& +& \frac{1}{C_4} \int_{D(\delta,x)} G_{s-t}\left((x-y) (A^T_s(x))^{-1} \right) \, dy \nonumber\\
& \le & c + c (s - t)^{\delta - (d + 3) \eps/\alpha}. \label{int2}
\end{eqnarray}

We also have
\begin{eqnarray*}
&& \left|p_{t,s}^y(x-y) - p_{t,s}^x(x-y)\right|\\
&& = \left| \frac{1}{|\det(A_s(y))|} G_{s-t}\left((x-y) (A^T_s(y))^{-1} \right)
- \frac{1}{|\det(A_s(x))|} G_{s-t}\left((x-y) (A^T_s(y))^{-1} \right)\right. \\
&& \left. +\frac{1}{|\det(A_s(x))|} G_{s-t}\left((x-y) (A^T_s(y))^{-1} \right)
- \frac{1}{|\det(A_s(x))|} G_{s-t}\left((x-y) (A^T_s(x))^{-1} \right) \right|\\
&& \le c |x - y|^{\gamma_1} G_{s-t}\left((x-y) (A^T_s(y))^{-1} \right)\\
&& + c \left| G_{s-t}\left((x-y) (A^T_s(y))^{-1} \right)
-  G_{s-t}\left((x-y) (A^T_s(x))^{-1} \right) \right|
\end{eqnarray*}
Note that for any $x \in \R^d$, $y \in D(\delta,x)$ and  $0 <s-t \le  \tau$ we have $|x - y|^{\gamma_1} \le c (s - t)^{\delta + (1 - \eps)/\beta}$. It follows that for any $x \in \R^d$ and  $0 <s-t \le  \tau$
\begin{equation}
\label{int3}
\int_{D(\delta,x)} \left|p_{t,s}^y(x-y) - p_{t,s}^x(x-y)\right| \, dy
\le c (s - t)^{\delta + (1 - \eps)/\beta} + c (s - t)^{\delta - (d + 3) \eps/\alpha}.
\end{equation}

Recall that in the case (A) we picked $\delta = \frac{(1-2\eps)\gamma_1}{\beta}$. Since  $\eps \le \eps_0\le \frac{\gamma_1 \alpha}{2(d+3)\beta}$ we have  $\frac{(d+3) \eps}{\alpha} \le \frac{\gamma_1}{2 \beta} = \frac{\delta}{2(1 - 2\eps)} < \delta$. Hence
$$\delta-\frac{(d+3) \eps}{\alpha}\ge \delta\left(1-\frac{1}{2(1 - 2\eps)}\right)\ge  (1 - 4\eps)\frac{(d+3) \eps}{\alpha}\ge \eps,$$
 since $\eps\le 1/8$. 

Recall that in the case (B) we picked  $\delta = (1-2\eps)((1 + \gamma_1)/\beta - 1/\alpha)$. Since
$\eps \le \eps_0 \le \frac{(1 + \gamma_1)/\beta - 1/\alpha}{2(d+3)/\alpha}$  we obtain
$\frac{(d+3) \eps}{\alpha} \le \frac{(1 + \gamma_1)/\beta - 1/\alpha}{2} =
\frac{\delta}{2(1 - 2\eps)} < \delta$ , since $\eps\le 1/8$. Hence, as in the case (A) we obtain
$$\delta-\frac{(d+3) \eps}{\alpha}\ge \delta\left(1-\frac{1}{2(1 - 2\eps)}\right)\ge  (1 - 4\eps)\frac{(d+3) \eps}{\alpha}\ge \eps.$$


Now (\ref{int1}), (\ref{int2}) imply (\ref{py_integral}). By (\ref{G_difference_inequality}) and (\ref{max1}) we get (\ref{py_px_difference}). (\ref{int1}), (\ref{int3}) imply (\ref{py_px_difference_int}).
\end{proof}

\begin{lemma}
\label{Gst_difference}

There exist $c= c(\bar{h}) $ such that for any  $0 \le t<s$ and $x \in \R^d$ we have
$$
\sup_{y \in \R^d} \left|G_{s - t}((x - y) (A_s^T(x))^{-1}) - G_{s - t}((x - y) (A_t^T(x))^{-1})\right|
\le c G_{s - t}(0) (s - t)^{\eps}.
$$
and
$$
\int_{\R^d} \left|G_{s - t}((x - y) (A_s^T(x))^{-1}) - G_{s - t}((x - y) (A_t^T(x))^{-1})\right| \, dy
\le c (s - t)^{\eps}.
$$
\end{lemma}
\begin{proof} We first prove the lemma under the assumption $0<s-t\le \tau_0$.
Let $\delta > 0$. Put
$$
\tilde{D}(\delta,x) =
\left\{w: \, |x - w| \le \frac{(s - t)^{\delta - \gamma_2}}{|M_{s - t}^{-1}|}\right\}.
$$
First, we consider  $y \in \tilde{D}(\delta,x)$. Then  we have
$$
\left|(x - y)(A_s^T(x))^{-1} - (x - y)(A_t^T(x))^{-1}\right| \le
\|A\| |x - y| (s - t)^{\gamma_2} \le
\frac{(s - t)^{\delta} \|A\|}{|M_{s - t}^{-1}|}.
$$
By the same arguments as in the proof of Lemma \ref{G_difference}, we get
\begin{eqnarray*}
&&  \left|G_{s - t}((x - y) (A_s^T(x))^{-1}) - G_{s - t}((x - y) (A_t^T(x))^{-1})\right|\\
&& \le c (s - t)^{\delta - \eps (1 + 1/\alpha)} G_{s - t}(0) e^{-|(x - y)(A_s^{-1}(x))^T M_{s - t}^{-1}|}.
\end{eqnarray*}
and
$$
\int_{\tilde{D}(\delta,x)} \left|G_{s - t}((x - y) (A_s^T(x))^{-1}) - G_{s - t}((x - y) (A_t^T(x))^{-1})\right| \, dy
\le c (s - t)^{\delta - (d + 3)\eps/\alpha}.
$$
Next, we estimate the expression $\left|G_{s - t}((x - y) (A_s^T(x))^{-1}) - G_{s - t}((x - y) (A_t^T(x))^{-1})\right|$ for $y \in \tilde{D}^c(\delta,x)$. In the case (A) we have $|M_{s - t}^{-1}|^{-1} = |M_{s - t}|$. Hence for $y \in \tilde{D}^c(\delta,x)$
$$
\frac{|x - y|}{|M_{s - t}|} \ge
(s - t)^{\delta - \gamma_2}.
$$
In the case (A) we will assume that  $\delta < \gamma_2$.

In the case (B), by (\ref{Mnorm}) and (\ref{Mnorm1}), for $y \in \tilde{D}^c(\delta,x)$
$$
\frac{|x - y|}{|M_{s - t}|} \ge
\frac{(s - t)^{\delta - \gamma_2}}{|M_{s - t}^{-1}| |M_{s - t}|} \ge
c (s - t)^{\delta - \gamma_2 + (1 - \eps)\left(\frac{1}{\alpha} - \frac{1}{\beta}\right)}.
$$
In the case (B) we will assume that  $\delta < \gamma_2 - \left(\frac{1}{\alpha} - \frac{1}{\beta}\right)$.

By the same arguments as in the proof of Lemma \ref{py_integrable}  we find $c = c(\bar{h})$ such that for $y \in \tilde{D}^c(\delta,x)$
$$
\left|G_{s - t}((x - y) (A_s^T(x))^{-1}) - G_{s - t}((x - y) (A_t^T(x))^{-1})\right|
\le c (s - t)^{\eps}\le c G_{s - t}(0)(s - t)^{\eps}
$$
and
$$
\int_{\tilde{D}^c(\delta,x)} \left|G_{s - t}((x - y) (A_s^T(x))^{-1}) - G_{s - t}((x - y) (A_t^T(x))^{-1})\right| \, dy
\le c (s - t)^{\eps}.
$$

In the case (A) we pick $\delta = (1-\eps)\gamma_2$. Since $\eps\le \eps_0 \le \frac{\gamma_2 \alpha}{2(d+3)}\le 1/4$ we have  $\frac{(d+3) \eps}{\alpha} \le \frac{\gamma_2}{2 } = \delta/2(1-\eps)$. Hence
$$\delta-\frac{(d+3) \eps}{\alpha}\ge \delta(1-\frac{1}{2(1 - \eps)})\ge  (1 - 2\eps)\frac{(d+3) \eps}{\alpha}\ge \eps,$$
so  we obtain the conclusion of the lemma in the case (A).
In the case (B) we pick $\delta = (1-\eps) \left(\gamma_2 - \left(\frac{1}{\alpha} - \frac{1}{\beta}\right)\right)$. Since
$\eps \le \eps_0 \le \frac{\gamma_2 - \left(\frac{1}{\alpha} - \frac{1}{\beta}\right)}{2(d+3)/\alpha}\le 1/4$ we get
$\frac{(d+3) \eps}{\alpha} \le \frac{\gamma_2 - \left(\frac{1}{\alpha} - \frac{1}{\beta}\right)}{2}
= \delta/2(1-\eps)$. As in the case (A) we obtain
$$\delta-\frac{(d+3) \eps}{\alpha}\ge \delta(1-\frac{1}{2(1 - \eps)})\ge  (1 - 2\eps)\frac{(d+3) \eps}{\alpha}\ge \eps.$$
This completes  the proof in the case $0<t-s\le \tau_0$.

In the case $0<t-s\ge \tau_0$ the conclusion is trivial since
$$
\sup_{y \in \R^d} \left|G_{s - t}((x - y) (A_s^T(x))^{-1}) - G_{s - t}((x - y) (A_t^T(x))^{-1})\right|
\le 2 G_{s - t}(0)
$$
and
$$
|\det (A_s^T(x))^{-1}|\int_{\R^d} \left|G_{s - t}((x - y) (A_s^T(x))^{-1}) - G_{s - t}((x - y) (A_t^T(x))^{-1})\right| \, dy
\le 2 .
$$

\end{proof}

\begin{lemma}
\label{q0_estimates_lemma} Suppose that  $0 <  s-t\le \tau$.
We have
\begin{equation*}
\int_{\R^d} |q_{t,s}^{(0)}(x,y)| \, dy \le c (s-t)^{-1 + \eps },\ x \in \R^d.
\end{equation*}
%
%
Moreover,
\begin{equation}
\label{q0_pointwise_estimate1}
|q_{t,s}^{(0)}(x,y)|  \le c (s-t)^{-1+\eps}G_{s-t}(0),\ x,y \in \R^d.
\end{equation}
The constant $c=c(\kappa(\tau))$.  If $\tau\le \tau_0$, then $c= c(\bar{h})$.
\end{lemma}
\begin{proof}
Let
$$
L_{t}^z f(x) = \sum_{k = 1}^d \mathrm{P.V.} \int_{\R} [f(x + u e_k A_t^T(z)+U_t(z,u e_k)) - f(x)] \, \nu_k(du).
$$
Recall that
$$
\left(\frac{\partial}{\partial t} + L_{t,s}^y\right) p_{t,s}^y(w) = 0.
$$
It follows that
\begin{eqnarray*}
&& |q_{t,s}^{(0)}(x,y)|\\
&& = \left|\left(\frac{\partial}{\partial t} + L_{t}\right) p_{t,s}^{(0)}(\cdot,y)(x)\right| \\
&& = \left|\left(\frac{\partial}{\partial t} + L_{t}^x\right) p_{t,s}^{y}(\cdot)(x-y)\right| \\
&& = \left|\left(-L_{t,s}^y + L_{t}^x\right) p_{t,s}^{y}(\cdot)(x-y)\right| \\
&& \le \left|\sum_{k = 1}^d \mathrm{P.V.} \int_{|u| <  R^{(k)}_{s-t}}
\left[p_{t,s}^{y}(x-y + u e_k A_t^T(x)) - p_{t,s}^{y}(x-y + u e_k A_s^T(y))\right]
\, \nu_k(du) \right|\\
&& +\left| \sum_{k = 1}^d \mathrm{P.V.} \int_{|u| < R^{(k)}_{s-t}}
\left[p_{t,s}^{y}(x-y + u e_k A_t^T(x) + U_t(x,u e_k)) \right. \right. \\
&& \quad \quad \quad \quad \quad \quad \quad \quad \quad \quad - \left. \left. p_{t,s}^{y}(x-y + u e_k A_t^T(x))\right]
\, \nu_k(du) \right|\\
&& +\left| \sum_{k = 1}^d \int_{|u| \ge R^{(k)}_{s-t}}
\left[p_{t,s}^{y}(x-y + u e_k A_t^T(x) + U_t(x,u e_k)) - p_{t,s}^{y}(x-y)\right]
\, \nu_k(du) \right|\\
&& = \text{I}(x,y) + \text{II}(x,y) + \text{III}(x,y).
\end{eqnarray*}

For the sake of simplicity we will denote $p(w)= p_{t,s}^y(w)$ and $z=x-y$.
To handle the term $\text{I}(x,y)$  we have to estimate
$$
 \mathrm{P.V.} \int_{|u| <  R^{(k)}_{s-t}}
\left[p(z + u e_k A_t^T(x)) - p(z + u e_k A_s^T(y))\right]
\, \nu_k(du).
$$
Because of the symmetry, we can re-write this integral as

$$
 \int_{|u| <  R^{(k)}_{s-t}}
\left[p(z+ u e_k A_t^T(x)) - p(z + u e_k A_s^T(y))\right.
\left.-\nabla p(z)(u e_k (A_t^T(x)-A_s^T(y))) \right]
\, \nu_k(du).
$$
We have
\begin{equation}
\begin{aligned}
&p(z + u e_k A_t^T(x)) - p(z + u e_k A_s^T(y))
-\nabla p(z)(u e_k (A_t^T(x)-A_s^T(y)))
\\&= \Big[\nabla p(z+\theta u e_k A_t^T(x)+ (1-\theta)u e_k A_s^T(y) )-\nabla p(z)\Big](u e_k (A_t^T(x)-A_s^T(y)))
\\&= \Big[\nabla p(z+u e_k A_s^T({y}) )- \nabla p(z+\theta u e_k A_t^T(x)+ (1-\theta)u e_k A_s^T(y) )\Big](u e_k (A_t^T(x)-A_s^T(y)))\\
&+ \Big[\nabla p(z+u e_k A_s^T({y}) )-\nabla p(z)\Big](u e_k (A_t^T(x)-A_s^T(y)))\\
&= \Delta_1+\Delta_2,
\end{aligned}\label{1}
\end{equation}
where $\theta\in [0,1]$.  Next,
$$\nabla p(z+u e_k A_s^T({y}) )- \nabla p(z+\theta u e_k A_t^T(x)+ (1-\theta)u e_k A_s^T(y) )=(\theta u e_k ( A_s^T(y)-A_t^T(x))\nabla^2 p(\xi),$$
where
\begin{equation*}
\xi = z +   \theta^* u e_k A_t^T(x)+ (1-\theta^*)u e_k A_s^T(y)),\ |u|\le R^{(k)}_{s-t}.
\end{equation*}
Troughout the whole proof we pick   $\delta = \frac{(1-2\eps)\gamma_1}{\beta}$ in the case (A) and
$\delta = (1-2\eps)((1 + \gamma_1)/\beta - 1/\alpha)$
in the case (B). Note that such choice of $\delta$ is dictated by  Lemma
\ref{py_integrable}, since we are going to use some arguments contained therein.

Let $|z|^{1+\gamma_1}\le  \frac{(s-t)^\delta} {|M^{-1}_{s-t}|}$, that is $y \in D(\delta,x)$.
 By Lemma  \ref{point} we have
\begin{equation*}
|\xi (A^T_s(y))^{-1} M^{-1}_{s-t}-  z(A^T_s(x))^{-1}M^{-1}_{s-t}| \le c.
\end{equation*}
Applying this  and  (\ref{hessianest}), we arrive at
$$
\left|\nabla^2 p(\xi)\right|
\le c \frac{|M^{-1}_{s-t}|^2}{(s-t)^{2\eps(1+1/\alpha)}}  e^{-|z (A_s^{-1}(x))^TM^{-1}_{s-t}|} G_{s-t}(0).
$$
This implies
\begin{eqnarray*}|\Delta_1|&\leq&  c  |u|^2 | A_s^T(y)-A_t^T(x)|^2
\frac{|M^{-1}_{s-t}|^2}{(s-t)^{2\eps(1+1/\alpha)}}  e^{-|z (A_s^{-1}(x))^TM^{-1}_{s-t}|} G_{s-t}(0)
\end{eqnarray*}
Next, by Lemma \ref{2moment},
$$
\int_{|u| < R^{(k)}_{s-t}}u^2\nu_k(du)\le
 (R^{(k)}_{s-t})^2  (s-t)^{\eps -1}
\le |M_{s-t}|^2 (s-t)^{\eps -1}.
$$
Hence,
$$
\begin{aligned}
\text{I}_1(x,y):&=\int_{|u| <R^{(k)}_{s-t}}|\Delta_1|\nu_k(du)\\&\le
c G_{s-t}(0)(|M^{-1}_{s-t}| |M_{s-t}|)^2 | A_s^T(y)-A_t^T(x)|^2
e^{-|z (A_s^{-1}(x))^TM^{-1}_{s-t}|}
    (s-t)^{ - \eps(1+2/\alpha)-1}.
\end{aligned}
$$

Now, let us estimate the second summand $\Delta_2$ in the right hand side of \eqref{1}.
%
Since  $p(w)= |\det((A^T_s(y))^{-1})|G_{s-t}\left(w (A^T_s(y))^{-1} \right), w\in \R^d$, we have
$$\nabla p(w) = |\det ((A^T_s(y))^{-1} )| \nabla G_{s-t}\left(w (A^T_s(y))^{-1} \right) (A^T_s(y))^{-1}. $$
%
Hence,
$$
\begin{aligned}
&  |\nabla p(z+u e_k A_s^T(y) )-\nabla p(z)|\\
&\le c  |\nabla G_{s-t}\left(z (A^T_s(y))^{-1}+u e_k\right)-\nabla G_{s-t}\left(z (A^T_s(y))^{-1}\right)|\\
&\leq  c  \left|\frac{\partial}{\partial w_k}\nabla G_{s-t}(\xi)\right||u|,
\end{aligned}
$$
where  $\xi = z (A^T_s(y))^{-1}+\theta u e_k, \ 0\le\theta\le 1$.
%
 By Lemma \ref{point}, we obtain
 \begin{equation*}
\label{xiest1}
|\xi  M^{-1}_{s-t}-  z(A^T_s(x))^{-1}M^{-1}_{s-t}| \le c.
\end{equation*}
Applying this  and  (\ref{hessianest}), we arrive at
\begin{equation*}
\left|\frac{\partial}{\partial w_k}\nabla G_{s-t}(\xi)\right|
\le c  e^{-|z (A_s^{-1}(x))^TM^{-1}_{s-t}|}
(s-t)^{-(2+2/\alpha)\eps}\frac1{R^{(k)}_{s-t}} |M^{-1}_{s-t}| G_{s-t}(0).
\end{equation*}
Then we  have
$$
\begin{aligned}
&|\Delta_2|=\left|\Big[\nabla p(z+u e_k A_s^T(y) )-\nabla p(z)\Big](u e_k (A_t^T(y)-A_t^T(x)))\right|
\\&\leq (s-t)^{-(2+2/\alpha)\eps}\frac1{R^{(k)}_{s-t}} |M^{-1}_{s-t}| G_{s-t}(0)|u|^2 \left|A_t^T(x)-A_s^T(y)\right|e^{-|z (A_s^{-1}(x))^TM^{-1}_{s-t}|}.
\end{aligned}
$$
Hence, since $\int_{|u| < R^{(k)}_{s-t}}u^2\nu_k(du)\le
   (R^{(k)}_{s-t})^2     (s-t)^{\eps -1}$, we have
$$
\begin{aligned}
\text{I}_2(x,y)&:=\int_{|u| < R^{(k)}_{s-t}}|\Delta_2|\nu_k(du)\\&\le
c G_{s-t}(0)|M^{-1}_{s-t}| R^{(k)}_{s-t} \left|A_t^T(x)-A_s^T(y)\right|
e^{-|z (A_s^{-1}(x))^TM^{-1}_{s-t}|}
    (s-t)^{-(2+2/\alpha)\eps} (s-t)^{\eps -1}\\
		&\le
c G_{s-t}(0)|M^{-1}_{s-t}| |M_{s-t}| \left|A_t^T(x)-A_s^T(y)\right|
e^{-|z (A_s^{-1}(x))^TM^{-1}_{s-t}|}
    (s-t)^{-(2+2/\alpha)\eps} (s-t)^{\eps -1}.
\end{aligned}
$$
We observe that, by (\ref{Mnorm1}),
$$|x-y|^{\gamma_1}\le \left( \frac{(s-t)^\delta} {|M^{-1}_{s-t}|}\right)^\frac{\gamma_1}{1+\gamma_1}\le
c(s-t)^{(\delta+(1-\eps)/\beta) \frac{\gamma_1}{1+\gamma_1}} \le c(s-t)^{((1-\eps)/\beta) \frac{\gamma_1}{1+\gamma_1}}.$$
 In the case (A) $|M^{-1}_{s-t}| |M_{s-t}|=1$,  hence
$$
\begin{aligned}|M^{-1}_{s-t}| |M_{s-t}||A_s^T(y)-A_t^T(x)|& \le  |A_s^T(y)-A_s^T(x)|+ | A_s^T(y)-A_t^T(y)|\\&\le c (|x-y|^{\gamma_1}+ (s-t)^{\gamma_2})\\
&\le c ( (s-t)^{\frac{1-\eps}\beta \frac{\gamma_1}{1+\gamma_1}}+ (s-t)^{\gamma_2})\\
&\le  c (s-t)^{(1-\eps)\rho},
\end{aligned}
$$
where $\rho = \min\{  \gamma_1/\beta(1+\gamma_1), \gamma_2  \}$.
In the case (B) we have , by  (\ref{Mnorm}) and(\ref{Mnorm1}),
$$
\begin{aligned}|M^{-1}_{s-t}| |M_{s-t}||A_s^T(y)-A_t^T(x)|& \le \left(|M^{-1}_{s-t}| |M_{s-t}|\right) |A_s^T(y)-A_s^T(x)|+ | A_s^T(y)-A_t^T(y)|\\&\le c \left(|M^{-1}_{s-t}| |M_{s-t}|\right)(|x-y|^{\gamma_1}+ (s-t)^{\gamma_2})\\
&\le c ( |M_{s-t}||M^{-1}_{s-t}|^{\frac{1}{1+\gamma_1}} (s-t)^{\delta\frac{\gamma_1}{1+\gamma_1}}+ |M^{-1}_{s-t}| |M_{s-t}|(s-t)^{\gamma_2})\\
&\le  (s-t)^{(1-\eps)(\frac1\beta - \frac1{ (1+\gamma_1)\alpha }) +\delta\frac{\gamma_1}{1+\gamma_1}}+
(s-t)^{(1-\eps)(\frac1\beta - \frac1{ \alpha}) +\gamma_2}\\
&\le  c (s-t)^{(1-\eps)\rho},
\end{aligned}
$$
where $\rho = \min\{ 1/\beta - 1/(1+\gamma_1)\alpha, \gamma_2+ 1/\beta - 1/\alpha   \}$.
This implies that for $y \in D(\delta,x)$,
\begin{eqnarray}\text{I}(x,y)&\le& c \text{I}_2(x, y)\nonumber \\&\le& c G_{s-t}(0)|M^{-1}_{s-t}| |M_{s-t}| \left|A_t^T(x)-A_s^T(y)\right|
e^{-|z (A_s^{-1}(x))^TM^{-1}_{s-t}|}
    (s-t)^{-(1+2/\alpha)\eps-1} \label{estIs0} \\&\le& c G_{s-t}(0)|M^{-1}_{s-t}| |M_{s-t}| \left|A_t^T(x)-A_s^T(y)\right|
(s-t)^{-(1+2/\alpha)\eps-1} e^{-\frac{|z|} {||A|||M_{s-t}|}}\nonumber\\&\le& c G_{s-t}(0)  (s-t)^{\rho(1-\eps)-(2+2/\alpha)\eps}
(s-t)^{\eps-1}e^{-\frac{|z|} {||A|||M_{s-t}|}} \nonumber \\&\le& c G_{s-t}(0)
(s-t)^{\eps-1}e^{-\frac{|z|} {||A|||M_{s-t}|}}, \label{estIs1}\end{eqnarray}
 provided
\begin{equation} \label{eps_cond}\eps\le \frac{\rho(1-\eps)}{2+2/\alpha}. \end{equation}
That is, when

$$0<\eps\le \frac{\min\{  \gamma_1/\beta(1+\gamma_1), \gamma_2  \}}{2+2/\alpha +\min\{  \gamma_1/\beta(1+\gamma_1), \gamma_2  \}}\quad \text{in the case (A)}$$
and
$$0<\eps\le \frac{\min\{ 1/\beta - 1/(1+\gamma_1)\alpha, \gamma_2+ 1/\beta - 1/\alpha   \}}{2+2/\alpha +\min\{ 1/\beta - 1/(1+\gamma_1)\alpha, \gamma_2+ 1/\beta - 1/\alpha   \}}\quad  \text{in the case (B)}.$$
Hence  (\ref{eps_cond}) holds for $\eps\le \eps_0$.
Next,
$$|A_s^T(y)-A_t^T(x)| \le |A_s^T(y)-A_s^T(x)|+ | A_s^T(y)-A_t^T(y)|\le c (|x-y|^{\gamma_1}+ (s-t)^{\gamma_2}).$$
%
%
By (\ref{estIs0}), Lemma \ref{int} used twice ( with $\rho =\gamma_1$ or  with $\rho =0$),  (\ref{det_G}) and finally  (\ref{Mnorm1}), we arrive at
\begin{eqnarray*}
&& \int_{D(\delta, x)}|I(x,y)|dy\\
&\le&  cG_{s-t}(0)det (M_{s-t})|M^{-1}_{s-t}| |M_{s-t}| (|M_{s-t}|^{\gamma_1} + (s-t)^{\gamma_2})(s-t)^{ - \eps(1+2/\alpha)-1}\\
&\le&  c(t-s)^{-\eps(2 +(d+2)/\alpha)}|M^{-1}_{s-t}| |M_{s-t}| (|M_{s-t}|^{\gamma_1} + (s-t)^{\gamma_2})(s-t)^{ \eps-1}\\
&\le&  c(t-s)^{-\eps(2 +(d+2)/\alpha)}|M^{-1}_{s-t}| |M_{s-t}| ((s-t)^{(1-\eps)\gamma_1/\beta} + (s-t)^{\gamma_2})
(s-t)^{ \eps-1}.
\end{eqnarray*}
In the case (A) we have  $|M^{-1}_{s-t}| |M_{s-t}|=1$, hence
\begin{equation} \label{Iestimate1}\int_{D(\delta, x)}|I(x,y)|dy \le c  (s-t)^{ -1+ \eps},\end{equation}
if  $ \eps\le \min\{\frac {\gamma_1}{\gamma_1+\beta(2+(d+2)/\alpha) }, \frac{\gamma_2}{2+(d+2)/\alpha}\}$, which is satisfied with our assumptions on $\eps$.
In the case (B) we have , by (\ref{Mnorm}) and (\ref{Mnorm1}) , $|M^{-1}_{s-t}| |M_{s-t}|\le c (s-t)^{(1-\eps)(-1/\alpha+1/\beta)}$, so
\begin{eqnarray}\int_{D(\delta, x)}|I(x,y)|dy &\le& c [ (t-s)^{-\eps(2 +(d+2)/\alpha)+(1-\eps)((1+\gamma_1)/\beta-1/\alpha)}\nonumber\\ &+& (t-s)^{-\eps(2 +(d+2)/\alpha)+(1-\eps)(1/\beta-1/\alpha)+\gamma_2}] (s-t)^{ -1+ \eps}\nonumber\\
&\le& c (s-t)^{ -1+ \eps}. \label{Iestimate2}
\end{eqnarray}
provided   $-\eps(2 +(d+2)/\alpha)+(1-\eps)((1+\gamma_1)/\beta-1/\alpha)\ge 0 $ and \\
$-\eps(2 +(d+2)/\alpha)+(1-\eps)(1/\beta-1/\alpha)+\gamma_2\ge 0$.
That is   $$\eps\le \min\left\{\frac {(1+\gamma_1)/\beta - 1/\alpha}{ (\gamma_1+1)/\beta +2 +(d+1)/\alpha},
\frac {(1/\beta - 1/\alpha+ \gamma_2)}{ 1/\beta +2 +(d+1)/\alpha} \right\}. $$
Again, this is true with our assumptions.

%



 Now we deal with the estimates of $\text{I}(x,y)$ over ${D^c(\delta,  x)}$. We note that our  assumptions yield that
in the case (A)  $\eps \le \frac{\gamma_1 \alpha}{2(d+3)\beta} $ ,
while
in the case (B)
$\eps \le \frac{(1 + \gamma_1)/\beta - 1/\alpha}{2(d+3)/\alpha}$.
We have
\begin{equation}
\label{difference}|p(z + u e_k A_t^T(x)) - p(z + u e_k A_s^T(y))-\nabla p(z)(u e_k (A_t^T(y)-A_t^T(x)))|\le c |\nabla^2 p(\xi)|u^2
\end{equation}
for
\begin{equation*}
\xi = z + \lambda ( \theta u e_k A_t^T(x)+ (1-\theta)u e_k A_s^T(y)),\ |u|\le R^{(k)}_{s-t},
\end{equation*}
with $\lambda, \theta \in [0,1]$.
We note that
\begin{eqnarray*}
|\xi- z| &\le&  \lambda (\theta|u|| e_k A_t^T(x)|+ (1-\theta)|u|| e_k A_t^T(y)|)\nonumber\\&\le&
  |u|||A||\nonumber\\
&\le&
   R^{(k)}_{s-t}||A||\\
&\le&
  |M_{s-t}|||A||.
\end{eqnarray*}
Hence, by (\ref{hessianest1}) and then by \eqref{Mnorm}, we get
\begin{eqnarray*}
\left|\nabla^2 p(\xi)\right|
&\le& c |M^{-1}_{s-t}|^2\exp \left(-\frac{|z|}{||A|||M_{s-t}|}\right)
 (s-t)^{-(2+\eps)/\alpha} G_{s-t}(0)\\
&\le& c \exp \left(-\frac{|z|}{||A|||M_{s-t}|}\right)
 (s-t)^{-4/\alpha} G_{s-t}(0).
\end{eqnarray*}
Combined with  (\ref{difference})  it yields
\begin{eqnarray*}
\nonumber
 \text{I}(x,y) &\le&
c \exp \left(-\frac{|z|}{||A|||M_{s-t}|}\right)
 (s-t)^{-4/\alpha} G_{s-t}(0)
\sum_{k = 1}^d \int_{|u| < R^{(k)}_{s-t}}
|u|^{2} \nu_k(du)
 \, du\\
&\le& c \exp \left(-\frac{|z|}{||A|||M_{s-t}|}\right)
 (s - t)^{-\frac{d+4}{\alpha}},
\end{eqnarray*}
since $\ G_{s-t}(0)\le c (s - t)^{-\frac{d}{\alpha}}$. Next, observe that
$\frac{|z|}{||A|||M_{s-t}|}\ge c(s - t)^{-\frac{\eps\gamma_1}{\beta}}$ for $y\in{D^c(\delta,  x)}$, which implies that
\begin{equation} \label{I(y)pointwise}\text{I}(x,y)\le c \exp \left(-\frac{|z|}{2||A|||M_{s-t}|}\right).\end {equation}
Using this  we obtain
\begin{equation}
\label{integral2}
\int_{{D^c(\delta,  x)}} \text{I}(x,y) \, dy \le c.
\end{equation}

Now we estimate $\text{II}(x,y)$. We have
\begin{equation}
\label{Lagrange}
\left|p(z + u e_k A_t^T(x) + U_t(x,u e_k)) - p(z + u e_k A_t^T(x))\right|
\le \left|\nabla p(\xi)U_t(x,u e_k)\right|,
\end{equation}
where
\begin{equation*}
\xi = z + u e_k A_t^T(x) + \lambda U_t(x,u e_k),\ |u|\le R^{(k)}_{s-t},
\end{equation*}
with $\lambda \in (0,1)$.

Let $|z|^{1+\gamma_1}\le  \frac{(s-t)^\delta} {|M^{-1}_{s-t}|}$. By Lemma  \ref{point}, we have
\begin{equation*}
|\xi (A^T_s(y))^{-1} M^{-1}_{s-t}-  z(A^T_s(x))^{-1}M^{-1}_{s-t}| \le c.
\end{equation*}
Applying this  and  (\ref{gradest}), we arrive at
\begin{equation}
\label{nablapxi}
\left|\nabla p(\xi)\right|
\le c \frac{|M^{-1}_{s-t}|}{(s-t)^{\eps(1+1/\alpha)}}  e^{-|z (A_s^{-1}(x))^TM^{-1}_{s-t}|} G_{s-t}(0).
\end{equation}
From Lemma \ref{2moment}  we infer that
$\int_{|u| < R^{(k)}_{s-t}}|u|^{\gamma_3}\nu_k(du)\le c
   (R^{(k)}_{s-t})^{\gamma_3}    (s-t)^{\eps -1}$. This combined with (\ref{Lagrange}) yield
\begin{equation}
\label{II(y)pointwise}
\text{II}(x,y)\le c \frac{|M^{-1}_{s-t}||M_{s-t}|^{\gamma_3}}{(s-t)^{\eps(1+1/\alpha)}}  e^{-|z (A_s^{-1}(x))^TM^{-1}_{s-t}|} G_{s-t}(0)(s-t)^{\eps -1}.
\end{equation}
We also note that 
\begin{equation}
\label{II(y)pointwise1}
\text{II}(x,y)\le c  G_{s-t}(0)(s-t)^{\eps -1}e^{-\frac{|z|} {||A|||M_{s-t}|}}.
\end{equation}
To prove it in the  case (A), using (\ref{Mnorm}), we observe that
	$$|M^{-1}_{s-t}||M_{s-t}|^{\gamma_3}= |M_{s-t}|^{\gamma_3-1}  \le c(s-t)^{(1-\eps) (\gamma_3-1)/\beta}.$$
	Hence, from (\ref{II(y)pointwise}) we obtain (\ref{II(y)pointwise1}) provided
	 $$ \eps\le \frac {\gamma_3-1}{\gamma_3+\beta(1+1/\alpha)},$$
which holds with our assumptions.
	To prove it in the  case (B) we observe   that, by (\ref{Mnorm}) and (\ref{Mnorm1}),
	$$|M^{-1}_{s-t}||M_{s-t}|^{\gamma_3}\le c(s-t)^{(1-\eps) (\gamma_3/\beta-1/\alpha)}.$$
	Hence, from (\ref{II(y)pointwise}) we obtain (\ref{II(y)pointwise1}) provided
	 $$ \eps\le \frac {\gamma_3-\beta/\alpha}{\gamma_3+\beta},$$
	which again holds in this case.

	Applying Lemma \ref{int} (with $\rho=0$), the estimate \eqref{II(y)pointwise}, (\ref{det_G}) and finally  (\ref{Mnorm1}),  we obtain
\begin{eqnarray*}\int_{D( \delta, x)}|\text{II}(x,y)|dy
&\le&  cG_{s-t}(0)det (M_{s-t})|M^{-1}_{s-t}| |M_{s-t}|^{\gamma_3} (s-t)^{-(1+1/\alpha)\eps} (s-t)^{  \eps-1}\\
&\le&  c(t-s)^{-\eps(1 +(d+1)/\alpha)}|M^{-1}_{s-t}| |M_{s-t}|^{\gamma_3}  (s-t)^{  \eps-1}\\
&\le&  c(t-s)^{-\eps(1 +(d+1)/\alpha)}|M^{-1}_{s-t}| |M_{s-t}|(s-t)^{(1-\eps)(\gamma_3-1)/\beta} (s-t)^{  \eps-1}.
\end{eqnarray*}
%
Then, by the same arguments as we applied to handle the term $\text{I}(x,y)$, we obtain
\begin{equation}\label{IIestimate}\int_{D( \delta, x)}|\text{II}(x,y)|dy \le c  (s-t)^{ -1+ \eps}\end{equation}
 in both cases: (A) (since $ \eps\le \frac {\gamma_3-1}{\gamma_3-1+\beta(1+(d+1)/\alpha)}$) and
(B) (since $ \eps\le \frac {\gamma_3-\beta/\alpha}{\gamma_3+\beta(1+d/\alpha)}$).
Moreover, again the same reasoning, as when $\text{I}(x,y)$ was explored, leads to
\begin{equation}\label{IIestimate1}\int_{D^c( \delta, x)}|\text{II}(x,y)|dy \le c \end{equation}
and
%
%
\begin{equation} \text{II}(x,y)\le c \exp \left(-\frac{|z|}{2||A|||M_{s-t}|}\right), \ y \in D^c(\delta,  x). \label{est_II}\end{equation}

By Lemma \ref{py_integrable},
\begin{eqnarray*}
\int_{\R^d} \text{III}(x,y) \, dy &\le& c \sum_{k = 1}^d \int_{|u| \ge  R^{(k)}_{s-t}}
\, \nu_k(du)
\le c \sum_{k = 1}^d h_k(R^{(k)}_{s-t})\\&=&c \sum_{k = 1}^d h_k(h^{-1}_k((s-t)^{-1 + \eps }))
= c d (s-t)^{-1 + \eps }.
\end{eqnarray*}

Using this, (\ref{Iestimate1}), (\ref{Iestimate2}), (\ref{integral2}), (\ref{IIestimate}) and (\ref{IIestimate1}) we get the first assertion of the lemma.

Finally it is clear that
$$\text{III}(x,y)\le c G_{s-t}(0)  \sum_{k = 1}^d \int_{|u| \ge R^{(k)}_{s-t}}
\, \nu_k(du) \le  c G_{s-t}(0)(s-t)^{-1 + \eps }.$$
This together with (\ref{estIs1}), (\ref{II(y)pointwise1}) and  (\ref{est_II}) prove the second assertion of the lemma. Finally, we remark that all the constants appearing in the above estimates   $c$ depned on $\tau $ through $\kappa(\tau)$ and for $\tau\le \tau_0$ the constants $c=c(\bar{h})$.

\end{proof}

\begin{lemma}
\label{q0_estimates_lemma_x} Assume \textbf{(I)} and  let $\tau>0$.   For $0 < t < s\le \tau$ and $y \in \R^d$ we have
\begin{equation*}
\int_{\R^d} |q_{t,s}^{(0)}(x,y)| \, dx \le c (s-t)^{-1 + \eps },
\end{equation*}
where $c= c(C_8, \kappa(\tau))$. If $\tau\le \tau_0$, then $c= c(C_8,\bar{h})$.
\end{lemma}
\begin{proof}
  The proof repeats partially the proof of Lemma \ref{q0_estimates_lemma}. Namely, take the decomposition
  $$
  q_{t,s}^{(0)}(x,y)=\text{I}(x,y)+\text{II}(x,y)+\text{III}(x,y)
  $$
  from this proof, and observe that literally the same estimates as in the above proof yield the required intergal-in-$x$ bound for the first two terms:
  $$
  \int_{\R^d} (|\text{I}(x,y)|+|\text{II}(x,y)|) \, dx \le c (s-t)^{-1 + \eps }.
  $$
  For the third term, we have to use the additional assumption \textbf{(I)}. Namely,
  $$
  \begin{aligned}
  \|\text{III}(\cdot, y)\|_{L_1}&\leq  \sum_{k = 1}^d \int_{|u| \ge R^{(k)}_{s-t}}
\int_{\R^d}\left[p_{t,s}^{y}(x-y + u e_k A_t^T(x) + U_t(x,u e_k)) + p_{t,s}^{y}(x-y)\right]
\, \nu_k(du)
\\&=\int_{z: z_k\ge R^{(k)}_{s-t}, k=1, \dots, d}
\left[\|T^{t,z}p_{t,s}^{y}(\cdot-y)\|_{L_1} + \|p_{t,s}^{y}(\cdot-y)\|_{L_1}\right]\,\mu(dz)
\\&\leq c \int_{z: z_k\ge R^{(k)}_{s-t}, k=1, \dots, d}
\mu(dz)
\\&\leq c(s-t)^{-1+\eps},
\end{aligned}
$$
where in the penultimate inequality we have used \textbf{(I)} and the identity
$$
\int_{\R^d} p_{t,s}^{y}(x,y) \, dx =1,
$$ which is easy to derive from the definition of $p_{t,s}^{y}(x,y)$.
\end{proof}

\begin{lemma}
\label{small_u}
Fix $\tau > 0$. For any $\xi \in (0,1]$, $\zeta > 0$,  $0 <s-t \le  \tau$, $x, y \in \R^d$, if $s - t \ge \xi$, then we have
\begin{eqnarray*}
&& \left|\sum_{k = 1}^d \int_{|u| <  R^{(k)}_{s-t} \wedge \zeta}
\left[p_{t,s}^{y}(x-y + u e_k A_t^T(x)) - p_{t,s}^{y}(x-y + u e_k A_s^T(y))\right]
\, \nu_k(du) \right|\\
&+& \left| \sum_{k = 1}^d \int_{|u| < R^{(k)}_{s-t} \wedge \zeta}
\left[p_{t,s}^{y}(x-y + u e_k A_t^T(x) + U_t(x,u e_k)) - p_{t,s}^{y}(x-y + u e_k A_t^T(x))\right]
\, \nu_k(du) \right|\\
&\le& c \sum_{k = 1}^d \int_{|u| <  R^{(k)}_{s-t} \wedge \zeta}
(|u|^2 + |u|^{\gamma_3})
\, \nu_k(du),
\end{eqnarray*}
where $c= c(\xi, \kappa(\tau))$.
\end{lemma}
\begin{proof}
The lemma follows from the estimates of $\Delta_1$, $\Delta_2$ (in the proof of Lemma \ref{q0_estimates_lemma}), (\ref{Lagrange}) and (\ref{nablapxi}).
\end{proof}

\begin{lemma}
\label{q0_new_estimates_lemma}
Fix $\tau > 0$.  We have
\begin{equation}
\label{q0_integral_estimate}
\lim_{r\to \infty}\sup_{x\in \R^d,  0<s-t<\tau }\int_{B^c(x,r)} |q_{t,s}^{(0)}(x,y)| \, dy=0
\end{equation}
and
\begin{equation}
\label{p0_integral_estimate}
\lim_{r\to \infty}\sup_{x\in \R^d,  0<s-t<\tau }\int_{B^c(x,r)} |p_{t,s}^{(0)}(x,y)| \, dy =0.
\end{equation}

\end{lemma}
\begin{proof}
Keeping the notation from Lemma \ref{q0_estimates_lemma} we have
$$
|q_{t,s}^{(0)}(x,y)| \le \text{I}(x,y) + \text{II}(x,y) + \text{III}(x,y).
$$
By \eqref{estIs1},  (\ref{I(y)pointwise}) and \eqref{II(y)pointwise1},  \eqref{est_II}
 we get for any $x,y\in \R^d$,
\begin{equation}
\label{I-IIpointwise_estimate}
\text{I}(x,y) + \text{II}(x,y) \le
c \exp\left(-\frac{|x-y|}{2\|A\| |M_{s-t}|}\right) (s-t)^{-d/\alpha +\eps-1}.
\end{equation}

By (\ref{bounded}) and (\ref{kernel_E}), for any $x \in \R^d$, $u \in \R$, $t > 0$ we have
\begin{equation}
\label{uEestimate}
\max_{1 \le k \le d} \left|u e_k A_t^T(x) + U_t(x,u e_k)\right| \le (dC_3+C_7)(|u|^{\gamma_3} \vee |u|).
\end{equation}
Put $r_0 = 2 (dC_3+C_7)$. We bound
$\text{III}(x,y)$ from above by
\begin{eqnarray*}
&& \left| \sum_{k = 1}^d
\int_{\frac{|x-y|^{1/\gamma_3}}{r_0^{1/\gamma_3}} \ge |u| \ge R^{(k)}_{s-t} }
\left[p_{t,s}^{y}(x-y + u e_k A_t^T(x) + U_t(x,u e_k)) - p_{t,s}^{y}(x-y)\right]
\, \nu_k(du) \right|
\\
&&+ \left| \sum_{k = 1}^d
\int_{|u| \ge \max\left(R^{(k)}_{s-t}, \frac{|x-y|^{1/\gamma_3}}{r_0^{1/\gamma_3}}\right)}
\left[p_{t,s}^{y}(x-y + u e_k A_t^T(x) + U_t(x,u e_k)) - p_{t,s}^{y}(x-y)\right]
\, \nu_k(du) \right|\\
&&= \text{IV}(x,y) + \text{V}(x,y).
\end{eqnarray*}


Assume now that $|x - y| \ge r_0 $.
When $u$ satisfies $|x-y|^{1/\gamma_3}/r_0^{1/\gamma_3} \ge |u| \ge R^{(k)}_{s-t}$ then, by (\ref{uEestimate}), we have
\begin{eqnarray*}
\left|x-y + u e_k A_t^T(x) + U_t(x,u e_k)\right|
&\ge& \left|x-y\right| - \left|u e_k A_t^T(x) + U_t(x,u e_k)\right|\\
&\ge& \left|x-y\right| - (r_0/2) (|u|^{\gamma_3} \vee |u|)\\
&\ge& \left|x-y\right| - (r_0/2) \frac{|x - y|}{r_0}\\
&=& \frac{|x - y|}{2}.
\end{eqnarray*}
Using this and Corollary \ref{general_matrix} we get
$$
\left|p_{t,s}^{y}(x-y + u e_k A_t^T(x) + U_t(x,u e_k)) - p_{t,s}^{y}(x-y)\right|
\le c G_{s-t}(0) \exp\left(-\frac{|x-y|}{2 \|A\| |M_{s-t}|}\right).
$$
It follows that for $|x-y| \ge r_0$ we have
\begin{eqnarray}
 \text{IV}(x,y)
&\le& c G_{s-t}(0)
\left(\sum_{k = 1}^d \int_{|u| \ge R_{s-t}^{(k)}} \, \nu_k(du)\right)
 \exp\left(-\frac{|x-y|}{2 \|A\| |M_{s-t}|}\right) \nonumber\\
&\le& c G_{s-t}(0) \exp\left(-\frac{|x-y|}{2 \|A\| |M_{s-t}|}\right)
,\label{vestimate1}
\end{eqnarray}
since for any $k \in \{1, \ldots, d\}$ we have
$\int_{|u| \ge R_{s-t}^{(k)}} \, \nu_k(du) \le h_k\left(R_{s-t}^{(k)}\right) = \frac{1}{(s-t)^{\eps - 1}}\le c$.

By elementary arguments for any $a, r > 0$ we have
$$
\int_{B^c(x,r)} e^{-a |x - y|} \, dy = \frac{c}{a^d} \int_{ar}^{\infty} e^{-v} v^{d - 1} \, dv \le \frac{c}{a^d} e^{- ar/2},
$$
where $c$ depends only on $d$.
Using this, \eqref{I-IIpointwise_estimate}, \eqref{vestimate1} and  and (\ref{Mnorm1}) we get for $r \ge r_0$
\begin{equation}
\label{I-IIestimate}
\int_{B^c(x,r)} (\text{I}(x,y) + \text{II}(x,y)+ \text{IV}(x,y)) \, dy
\le c (s-t)^{-d/\alpha +\eps-1} |M_{s - t}|^d \exp\left(\frac{- r}{4 |M_{s - t}| \|A\|}\right)
\le c e^{-c_1 r}.
\end{equation}

By Lemma \ref{py_integrable}, we have
\begin{eqnarray}
\nonumber
\int_{\R^d} \text{IV}(x,y) \, dy
&\le& c \sum_{k = 1}^d
\int_{|u| \ge \max\left(R^{(k)}_{s-t}, r^{1/\gamma_3}/r_0^{1/\gamma_3}\right)}
\, \nu_k(du) \\
&\le& c \sum_{k = 1}^d h_k\left(\frac{r^{1/\gamma_3}}{r_0^{1/\gamma_3}}\right).\label{IVestimate}
\end{eqnarray}
Since  $\lim_{r\to \infty}h_k(r)=0$, the first assertion of the lemma follows from (\ref{I-IIestimate}) and  (\ref{IVestimate}).

The proof of the second follows easily from \eqref{densityest1}.
\end{proof}

Put
$$
W = \{(t, s):\, t, s \in [0,\infty), \, t < s\}.
$$
\begin{lemma}
\label{pts_continuous}
The function $W\times\R^d\times\R^d \ni (t, s, x, y) \to p_{t, s}^{y}(x)$ is continuous as well as  the function $W\times\R^d\times\R^d \ni (t, s, x, y) \to q_{t,s}^{(0)}(x,y)$.
\end{lemma}
\begin{proof}
The first assertion  follows from Lemma \ref{gu(0)estimate} and continuity of the map  $(s,x)\mapsto A_s(x)$.
Recall that $q_{t,s}^{(0)}(x,y)$ is equal to
\begin{eqnarray*}
&& \sum_{k = 1}^d \int_{|u| <  R^{(k)}_{s-t}}
\left[p_{t,s}^{y}(x-y + u e_k A_t^T(x)) - p_{t,s}^{y}(x-y + u e_k A_s^T(y))\right]
\, \nu_k(du)\\
&& + \sum_{k = 1}^d \int_{|u| <  R^{(k)}_{s-t}}
\left[p_{t,s}^{y}(x-y + u e_k A_t^T(x) + U_t(x,u e_k)) - p_{t,s}^{y}(x-y + u e_k A_s^T(x))\right]
\, \nu_k(du) \\
&& + \sum_{k = 1}^d \int_{|u| \ge R^{(k)}_{s-t}}
\left[p_{t,s}^{y}(x-y + u e_k A_t^T(x) + U_t(x,u e_k)) - p_{t,s}^{y}(x-y)\right]
\, \nu_k(du).
\end{eqnarray*}
Hence, the second assertion of the lemma follows from the first, Lemma \ref{small_u}, (\ref{densityest1}) and the bounded convergence theorem.
\end{proof}
\begin{lemma}
\label{pts0_delta}
For any $f \in C_{\infty}(\R^d)$, $t_0 \ge 0$ we have
$$
\lim_{W \ni (t,s) \to (t_0,t_0)} \|P_{t,s}^{(0)} f - f\|_{\infty} = 0.
$$
\end{lemma}
\begin{proof}
Note that for any $x \in \R^d$, $0 < t < s < \infty$ we have $\int_{\R^d} p_{t,s}^x(x - y) \, dy = 1$. Using this, (\ref{py_px_difference_int}) and (\ref{densityest1}) we easily obtain the assertion of the lemma.
\end{proof}
\begin{lemma}
\label{qts0_delta}
For any $f \in C_{\infty}(\R^d)$, $0 \le t_0 < s_0$ we have
$$
\lim_{W \ni (t,s) \to (t_0,s_0)} \|Q_{t,s}^{(0)} f - Q_{t_0,s_0}^{(0)}f\|_{\infty} = 0
$$and
$$
\lim_{W \ni (t,s) \to (t_0,s_0)} \|P_{t,s}^{(0)} f - P_{t_0,s_0}^{(0)}f\|_{\infty} = 0.
$$
\end{lemma}
\begin{proof} We give a detailed  proof of the first statement and only  a sketch for the second.

 By Lemma \ref{q0_new_estimates_lemma}, it is enough to prove the lemma for $f$ with compact support.
We note that the function $(W\times\R^d) \ni (t, s, x) \to Q_{t,s}^{(0)}f(x)$ is continuous. This follows from \eqref{q0_pointwise_estimate1}, Lemma \ref{pts_continuous} and the bounded convergence theorem.
 Let $r>0$.
Hence it is uniformly continuous on
$$
\{t, s, x: \, t \ge 0; \, |x|\le r; \,  |t-t_0|, |s-s_0| \le |s_0-t_0|/3 \}.
$$
It follows that
$$
\lim_{W \ni (t,s) \to (t_0,s_0)} \sup_{|x|\le r} |Q_{t,s}^{(0)} f(x) - Q_{t_0,s_0}^{(0)}f(x)| = 0.
$$
Let $r$ be so large that the support of $f$ is contained in $B(0, r/2)$.
Next, we have
 $$\sup_{|x|\ge r,  0<s-t<2(s_0-t_0) }|Q_{t,s}^{(0)} f(x)|\le \|f\|_{\infty}\sup_{x\in \R^d,  0<s-t<2(s_0-t_0) }\int_{B^c(x,r/2)} |q_{t,s}^{(0)}(x,y)|dy.$$
Hence
$$
\limsup_{W \ni (t,s) \to (t_0,s_0)}  \|Q_{t,s}^{(0)} f - Q_{t_0,s_0}^{(0)}f\|_{\infty}  \le 2\|f\|_{\infty}\sup_{x\in \R^d,  0<s-t<2(s_0-t_0) }\int_{B^c(x,r/2)} |q_{t,s}^{(0)}(x,y)|dy,$$
which converges to $0$, if $r\to \infty$, by Lemma \ref{q0_new_estimates_lemma}. This completes the proof of the first assertion.

Finally, we remark that  the function $W\times\R^d \ni (t, s, x) \to P_{t,s}^{(0)}f(x)$ is continuous, due to Lemma  \ref{pts_continuous}. Next, similarly as above, we apply  \eqref{p0_integral_estimate} to complete the proof of the second assertion.
\end{proof}
\begin{lemma}
For any  $0 <s-t \le  \tau$ and $x, y \in \R^d$ such that $|x - y| \le (s - t)^{1/\alpha}$ we have
\begin{equation}
\label{xydifference1}
\int_{\R^d} \left|p_{t,s}^z(x-z) - p_{t,s}^z(y-z)\right| \, dz
\le c |x - y| (s - t)^{-1/\alpha - (d + 2)\eps/\alpha},
\end{equation}
where $c = c(\kappa(\tau))$.
\end{lemma}
\begin{proof}
We have
\begin{eqnarray*}
&& p_{t,s}^z(x-z) - p_{t,s}^z(y-z)\\
&& = \frac{1}{|\det(A_s(z))|}
\left[G_{s - t}((x - z)(A_s^{-1}(z))^T) - G_{s - t}((y - z)(A_s^{-1}(z))^T)\right]\\
&& = \frac{1}{|\det(A_s(z))|} \nabla G_{s-t}(\xi)\left[(x - y)(A_s^{-1}(z))^T\right],
\end{eqnarray*}
where
$\xi = (\theta(x - z) + (1 - \theta)(y - z)) (A_s^{-1}(z))^T$, $0 \le \theta \le 1$.
By Lemma \ref{Gest} and then Lemma
\ref{scal_h}, we obtain
\begin{eqnarray*}
\left|p_{t,s}^z(x-z) - p_{t,s}^z(y-z)\right|
&\le& c |x - y| G_{s - t}(0) \frac{1}{h_{min}^{-1}(1/(s - t))} \frac{1}{(s - t)^{\eps}}
e^{-|\xi M_{s - t}^{-1}|}\\
&\le& c |x - y| G_{s - t}(0)  \frac{1}{(s - t)^{\eps+1/\alpha}}
e^{-|\xi M_{s - t}^{-1}|}.
\end{eqnarray*}
We have
$$
|\xi M_{s - t}^{-1}| \ge
|(x - z) (A_s^{-1}(z))^T M_{s - t}^{-1}|
- |\xi M_{s - t}^{-1} - (x - z) (A_s^{-1}(z))^T M_{s - t}^{-1}|.
$$
Since $|x - y| \le (s - t)^{1/\alpha}$, using (\ref{Mnorm1}), we get
\begin{eqnarray*}
|\xi M_{s - t}^{-1} - (x - z) (A_s^{-1}(z))^T M_{s - t}^{-1}|
&\le&
|(-(1-\theta)x + (1-\theta)y) (A_s^{-1}(z))^T M_{s - t}^{-1}|\\
&\le& |x - y| \|A\| |M_{s - t}^{-1}|\\
&\le& c |x - y| (s - t)^{-1/\alpha  + \eps/\alpha}\\
&\le& c.
\end{eqnarray*}
We pick  $\delta > 0$ in the same way as in  Lemma
\ref{py_integrable}. By the same arguments as in the proof of Lemma \ref{G_difference} for $z \in D(\delta,x)$ we have
$$
|(x - z) (A_s^{-1}(z))^T - (x - z) (A_s^{-1}(x))^T| \le \|A\| |x - z|^{1 + \gamma_1}
\le (s - t)^{\delta}  \|A\| \frac{1}{|M_{s - t}^{-1}|}.
$$
Hence
$$
|(x - z) (A_s^{-1}(z))^T M_{s - t}^{-1} - (x - z) (A_s^{-1}(x))^T M_{s - t}^{-1}|
\le (s - t)^{\delta} |M_{s - t}^{-1}| \|A\| |M_{s - t}^{-1}|^{-1}
= (s - t)^{\delta} \|A\|.
$$
Therefore for $z \in D(\delta,x)$ we have
$$
\left|p_{t,s}^z(x-z) - p_{t,s}^z(y-z)\right|
\le c |x - y| G_{s - t}(0) \frac{1}{(s - t)^{\eps+1/\alpha}}
e^{-|(x - z) (A_s^{-1}(x))^T M_{s - t}^{-1}|}.
$$
Using the above estimates, Lemma \ref{int} and Lemma
\ref{Gest} we get
\begin{eqnarray*}
&& \int_{D(\delta,x)} \left|p_{t,s}^z(x-z) - p_{t,s}^z(y-z)\right| \, dz\\
&& \le c |x - y| G_{s - t}(0) \frac{1}{(s - t)^{\eps+1/\alpha}}
\int_{D(\delta,x)} e^{-|(x - z) (A_s^{-1}(x))^T M_{s - t}^{-1}|} \, dz\\
&& \le c |x - y| G_{s - t}(0) \frac{1}{(s - t)^{\eps+1/\alpha}}
\det(M_{s - t})\\
&& \le \frac{c |x - y|}{(s - t)^{1/\alpha + \eps}}
\prod_{i = 1}^d \frac{h_i^{-1}(1/(s - t)^{1 - \eps})}{h_i^{-1}(1/(s - t))}.
\end{eqnarray*}
By  Corollary  \ref{h_inv0} this is bounded from above by
$$
c |x - y| (s - t)^{- 1/\alpha - (d + 2) \eps/\alpha}.
$$
For $z \in D^c(\delta, x)$ we have
\begin{eqnarray*}
&& \left|p_{t,s}^z(x-z) - p_{t,s}^z(y-z)\right|\\
&& \le c |x - y| G_{s - t}(0) \frac{1}{(s - t)^{\eps+1/\alpha}}
e^{-|(x - z) (A_s^{-1}(z))^T M_{s - t}^{-1}|}\\
&& \le c |x - y| G_{s - t}(0) \frac{1}{(s - t)^{\eps+1/\alpha}}
e^{-\frac{|x - z|}{\|A\| |M_{s - t}|}}.
\end{eqnarray*}
Using (\ref{int1}) we infer that  there exists $c $ such that
\begin{eqnarray*}
&& \int_{D^c(\delta, x)} \left|p_{t,s}^z(x-z) - p_{t,s}^z(y-z)\right| \, dz\\
&& \le c |x - y| G_{s - t}(0) \frac{1}{(s - t)^{\eps+1/\alpha}}
\int_{D^c(\delta, x)} e^{-\frac{|x - z|}{\|A\| |M_{s - t}|}} \, dz\\
&& \le c  |x - y|,
\end{eqnarray*}
which finishes the proof of (\ref{xydifference1}).
\end{proof}

\begin{proof}[Proof of Theorem \ref{Holder_thm_new}] Let  $0 < \gamma < \gamma' < \alpha$, $\gamma \le 1$. We pick $\eps =\min\{ \eps_0, \frac{\gamma'-\gamma }{\gamma(d+2)}\}$. Then for any
 $0 <s-t \le  \tau$ and $x, y \in \Rd$ such that $|x - y| \le (s - t)^{1/\alpha}$ we have
\begin{equation}
\label{xydifference2}
\int_{\R^d} \left|p_{t,s}^z(x-z) - p_{t,s}^z(y-z)\right| \, dz
\le c |x - y|^{\gamma} (s - t)^{-\gamma'/\alpha},
\end{equation}
where $c = c(\gamma,\gamma',\kappa(\tau))$. To prove   (\ref{xydifference2}) we observe that our choice of  $\eps \in (0,\eps_0]$ yields
$1 + (d + 2) \eps \le  \gamma'/\gamma$.
Hence,
by (\ref{xydifference1}), we get
$$
\left(\int_{\R^d} \left|p_{t,s}^w(x-w) - p_{t,s}^w(y-w)\right| \, dw\right)^{\gamma}
\le c |x - y|^{\gamma} (s - t)^{-\frac{\gamma}{\alpha}(1 + (d + 2) \eps)}
\le c |x - y|^{\gamma} (s - t)^{-\gamma'/\alpha}.
$$
On the other hand, by (\ref{py_integral}), we obtain
$$
\left(\int_{\R^d} \left|p_{t,s}^w(x-w) - p_{t,s}^w(y-w)\right| \, dw\right)^{1 - \gamma}
\le c.
$$
The last two estimates imply (\ref{xydifference2}).

If $|x - y| \ge ((s - t)/2)^{1/\alpha}$ then the assertion of the theorem is trivial, so we may assume that $|x - y| < ((s - t)/2)^{1/\alpha}$.
We have
\begin{eqnarray*}
&& \left|P_{t,s} f(x) - P_{t,s}f(y)\right| =
\left| \int_{\R^d} \left(p_{t,s}^z(x-z) - p_{t,s}^z(y-z)\right) f(z) \, dz \right. \\
&& \left. + \int_t^s \int_{\R^d} \left(p_{t,r}^w(x-w) - p_{t,r}^w(y-w)\right)
\int_{\R^d} q_{r,s}(w,z) f(z) \, dz \, dw \, dr \right|.
\end{eqnarray*}
By (\ref{xydifference2}) and (\ref{para-e30a}), this is bounded from above by
\begin{eqnarray}
\nonumber
&& c |x - y|^{\gamma} (s-t)^{-\gamma'/\alpha} \|f\|_{\infty}\\
\label{Ptsxy}
&& + c \|f\|_{\infty}
\int_t^s \int_{\R^d} \left|p_{t,r}^w(x-w) - p_{t,r}^w(y-w)\right|
(s - r)^{\eps \alpha - 1} \, dw \, dr.
\end{eqnarray}
Recall that we assumed $|x - y| < ((s - t)/2)^{1/\alpha}$. Let us denote
\begin{eqnarray*}
&& \int_t^s \int_{\R^d} \left|p_{t,r}^w(x-w) - p_{t,r}^w(y-w)\right|
(s - r)^{\eps \alpha - 1} \, dw \, dr\\
&& = \int_t^{t + |x - y|^{\alpha}} \ldots
+ \int_{t + |x - y|^{\alpha}}^{t + (s - t)/2} \ldots
+ \int_{t + (s - t)/2}^{s - t} \ldots\\
&& = \text{I} + \text{II} + \text{III}.
\end{eqnarray*}

By (\ref{py_integral}) and our assumption $|x - y|^{\alpha} < (s - t)/2$ we get
$$
\text{I}
\le c \int_t^{t + |x - y|^{\alpha}} \left(\frac{s - t}{2}\right)^{\eps \alpha - 1} \, dr
\le c \frac{|x - y|^{\alpha}}{s - t}
\le c \left(\frac{|x - y|^{\alpha}}{s - t}\right)^{\gamma/\alpha}
\le c |x - y|^{\gamma} (s-t)^{-\gamma'/\alpha}.
$$
By (\ref{xydifference2}) we get
$$
\text{II}
\le c
 |x - y|^{\gamma} (r - t)^{-\gamma'/\alpha} \, dr
\le c |x - y|^{\gamma} (s-t)^{-\gamma'/\alpha}.
$$
Again by (\ref{xydifference2}) we obtain
$$
\text{III}
\le c
\int_{t + (s - t)/2}^{s - t} |x - y|^{\gamma} (r - t)^{-\gamma'/\alpha} (s - r)^{\eps \alpha - 1} \, dr
\le c |x - y|^{\gamma} (s-t)^{-\gamma'/\alpha}.
$$
By (\ref{Ptsxy}) and the estimates of $\text{I}$, $\text{II}$, $\text{III}$ we obtain the assertion of the theorem.
\end{proof}

\begin{lemma}
\label{p0_ptilde_difference}
For any  $0 <s-t \le  \tau$ and $x \in \R^d$  we have
\begin{equation}
\label{p0_ptilde_difference_max}
\sup_{y \in \R^d} \left|p_{t,s}^{(0)}(x,y) - \tilde{p}_{t,s}(x,y)\right| \le c G_{s - t}(0) (s - t)^{\eps}
\end{equation}
and
\begin{equation}
\label{p0_ptilde_difference_int}
\int_{\R^d}\left|p_{t,s}^{(0)}(x,y) - \tilde{p}_{t,s}(x,y)\right| \, dy \le c (s - t)^{\eps}.
\end{equation}
The contant $c= c( \kappa(\tau))$. If $\tau\le \tau_0$, then $c= c(\bar{h})$.
\end{lemma}
\begin{proof}
Let  $0 <s-t \le  \tau$, $x \in \R^d$ be arbitrary. We have
\begin{eqnarray*}
\left|p_{t,s}^{(0)}(x,y) - \tilde{p}_{t,s}(x,y)\right|
&=&
\left|p_{t,s}^y(y - x) - \tilde{p}_{t,s}(x,y)\right|\\
&\le&
\left|p_{t,s}^y(y - x) - p_{t,s}^x(y - x)\right|
+
\left|p_{t,s}^x(y - x) - \tilde{p}_{t,s}(x,y)\right|\\
&=&  \text{I}_1 +\text{I}_2.
\end{eqnarray*}
We also have
\begin{eqnarray*}
\text{I}_2 &\le&
\left|
\frac{1}{|\det A_s(x)|} G_{s - t}((A_s(x))^{-1}(y - x)) - \frac{1}{|\det A_t(x)|} G_{s - t}((A_s(x))^{-1}(y - x))
\right|\\
&+&
\frac{1}{|\det A_t(x)|} \left|
 G_{s - t}((A_s(x))^{-1}(y - x)) - G_{s - t}((A_t(x))^{-1}(y - x))
\right|\\
&+&
\frac{1}{|\det A_t(x)|} \left|
 G_{s - t}((A_t(x))^{-1}(y - x)) - \tilde{G}_{s - t}((A_t(x))^{-1}(y - x))
\right|\\
&=&
\text{I}_3 + \text{I}_4 + \text{I}_5.
\end{eqnarray*}
It remains to justify $\sup_{y \in \R^d} \text{I}_i \le c G_{s - t}(0) (s - t)^{\eps}$ and $\int_{\R^d} \text{I}_i \, dy \le c (s - t)^{\eps}$ for $i \in \{1,2, 3,4,5\}$ and some $c > 0$ for some $c=c(\kappa(\tau))$ in the general case or $c=c(\bar{h})$ if $t<s\le\tau_0$.

By Lemma \ref{py_integrable} we get such estimates for $\text{I}_1$. By (\ref{bounded}), (\ref{determinant}) and (\ref{Lipschitz}) we get the estimates for $\text{I}_3$. By Lemma \ref{Gst_difference} and (\ref{determinant}) we obtain such estimates for $\text{I}_4$. Analogous estimates of $\text{I}_5$ follow from Lemma \ref{gu_difference} and definitions of $G_{s - t}$, $\tilde{G}_{s - t}$.
\end{proof}

\appendix

\section{Estimates for Example \ref{ex1} and Example \ref{ex2}}\label{sA}

In this section we prove the two inequalities, which were stated and used in  Example \ref{ex1} and Example \ref{ex2}.
\begin{proof}[Proof of \eqref{h-discrete}] For $\rho_{k+1}<|y|\leq \rho_k$ we have $\rho_k\leq c^{-1}|y|$, hence
$$\begin{aligned}
h(r)&=r^{-2}\sum_{k: \rho_k\leq r}\rho_k^2\nu(\rho_{k+1}<|y|\leq \rho_k)+\sum_{k:\rho_k>r}\nu(\rho_{k+1}<|y|\leq \rho_k)
\\&\leq c^{-2}r^{-2}\int (1\wedge(|y|^2r^{-2})\nu(dy)= c^{-2} \frac{4c_\alpha}{\alpha(2-\alpha)} r^{-\alpha},
\end{aligned}
$$
which proves the upper bound in \eqref{h-discrete}. Similarly, we have
$$
\begin{aligned}
h(r)&\geq \int_{|y|\leq \rho_1}(1\wedge(|y|^2r^{-2})\nu(dy)\geq Br^{-\alpha}, \quad r\in (0,1],
\end{aligned}
$$
which proves the lower bound.
\end{proof}
\begin{proof}[Proof of \eqref{on-diagonal-est}] Without loss of generality we can take $t=0, x=0,$ then
$$
X_s=\int_0^sA(r)\mathbf{e}_1\, dZ_r^1+\int_0^sA(r)\mathbf{e}_2\, dZ_r^2
$$
The characteristic function of $X_t$ has the form
$$
\phi_s^X(z)=\exp\left\{-\int_0^s(|(A(r)z)_1|^{\alpha_1}+|(A(r)z)_2|^{\alpha_2})\, dr\right\},
$$
and thus the distribution density equals
$$
p_s^X(x)=\frac{1}{(2\pi)^2}\int_{\R^2}\exp\left\{-ix\cdot z-\int_0^s(|(A(r)z)_1|^{\alpha_1}+|(A(r)z)_2|^{\alpha_2})\, dr\right\}\, dz.
$$
In particular,
$$
p_s^X(0)=\frac{1}{(2\pi)^2}\int_{\R^2}\exp\left\{-\int_0^s(|(A(r)z)_1|^{\alpha_1}+|(A(r)z)_2|^{\alpha_2})\, dr\right\}\, dz,
$$
below we will show that the latter integral exists.

Lets estimate from below $$
\int_0^s(|(A(r)z)_1|^{\alpha_1}+|(A(r)z)_2|^{\alpha_2})\, dr\geq \int_0^s|(A(r)z)_2|^{\alpha_2}\, dr
$$
We recall that
$$
(A(r)z)_2=r^\gamma z_1+z_2, \quad z=(z_1, z_2),
$$ and perform case study.

\emph{Case 1:} $|z_2|>\frac{s^\gamma}2|z_1|$. Then
$$
|r^\gamma z_1+z_2|\geq \frac14|z_2|\hbox{ if } r\in [0,\frac{s}{4^{1/\gamma}}] \hbox{ and } |r^\gamma z_1+z_2|\geq0 \hbox{ othwerwize,}
$$
which gives
$$
\int_0^s|(A(r)z)_2|^{\alpha_2}\, dr\geq cs |z_2|^{\alpha_2}.
$$

\emph{Case 2:} $|z_2|\leq \frac{s^\gamma}2|z_1|$. Consider two intervals $I(s)=[\frac{s}{2},\frac{3s}{4}), J(s)= [\frac{3s}{4}, s].$ At least one of these intervals is free from the roots of the function $r\mapsto  |r^\gamma z_1+z_2|$, and this function depends on $v=r^{\gamma}$ linearly  with the slope $\pm|z_1|$. Since the values of this function in the endpoints are positive, this yields that, at least on the half of the interval,
$$
|r^\gamma z_1+z_2|\geq c s^\gamma|z_1|,
$$
which gives
$$
\int_0^s|(A(r)z)_2|^{\alpha_2}\, dr\geq cs^{1+\alpha_2\gamma} |z_1|^{\alpha_2}.
$$

Now we can complete the estimate of $p^Y(0)$. We have
$$\begin{aligned}
p_s^X(0)&\leq \frac{1}{(2\pi)^2}\int_{\R^2}\exp\left\{-\int_0^t|(A(s)z)_2|^{\alpha_2}\, ds\right\}\, dz
\\&\leq \frac{1}{(2\pi)^2} \int_{|z_2|> \frac{s^\gamma}2|z_1|}\exp\left\{-cs |z_2|^{\alpha_2}\right\}\, dz
\\&+ \frac{1}{(2\pi)^2}\int_{|z_2|\leq \frac{s^\gamma}2|z_1|}\exp\left\{-c s^{1+\alpha_2\gamma} |z_1|^{\alpha_2}\right\}\, dz=:I_1+I_2.
\end{aligned}
$$
Since
$$
I_1=\frac{4}{(2\pi)^2s^\gamma}\int_{\Re}|z_2|\exp\left\{-cs |z_2|^{\alpha_2}\right\}\, dz_2=\Big|_{s^{1/\alpha_2} z_2=v}=Cs^{-\gamma-{2/\alpha_2}},
$$
$$
\begin{aligned}
I_2&=\frac{s^\gamma}{(2\pi)^2}\int_{\Re}|z_1|\exp\left\{-c s^{1+\alpha_2\gamma} |z_1|^{\alpha_2}\right\}\, dz_1=\Big|_{s^{(1+\alpha_2\gamma)/\alpha_2} z_1=v}
\\&=Cs^\gamma\cdot t^{-2(1+\alpha_2\gamma)/\alpha_2}= Cs^{-\gamma-{2/\alpha_2}},
\end{aligned}$$
this completes the proof of \eqref{on-diagonal-est}.\end{proof}

\end{document}